\documentclass[12pt]{article}
\usepackage{amsmath,amsthm}
\usepackage{amsfonts,mathrsfs}
\usepackage[colorlinks,anchorcolor=blue,citecolor=blue]{hyperref}
\usepackage[authoryear]{natbib}

\allowdisplaybreaks
\theoremstyle{plain}
\newtheorem{thm}{\noindent\bf Theorem}%[section]%[subsection]
\newtheorem{cor}{\noindent\bf Corollary}
\newtheorem{lem}{\noindent\bf Lemma}
\newtheorem{prop}{\noindent\bf Proposition}
\theoremstyle{remark}
\newtheorem{rmk}{\noindent \bf \textit{Remark}}

\newcommand{\D}{\displaystyle}

\newcommand{\bP}{\mathbb{P}}
\newcommand{\bq}{\mathbb{Q}}
\newcommand{\bE}{\mathbb{E}}

\usepackage{ulem}
\newcounter{num}
\newcommand{\wo}{W^{(\omega)}}

\newcommand{\ho}{H^{(\omega)}}
\usepackage{enumerate}
\usepackage{ushort} % \ushort

\newcommand{\aspzero}{\hyperlink{asp:0}{$\mathbb{H}_{0}$}}
\newcommand{\aspone}{\hyperlink{asp:1}{$\mathbb{H}_{1}$}}
\newcommand{\asptwo}{\hyperlink{asp:2}{$\mathbb{H}_{2}$}}

\begin{document}
%\linenumbers

\title{On the explosion of a class of continuous-state nonlinear branching processes}

%\author{Bo Li and Xiaowen Zhou}

\author{
	{\sc Bo Li}\\[.01in]
	School of Mathematics and LPMC\\
	Nankai University, Tianjin 300071, PR China\\
	[.01in]\\
	{\sc Xiaowen Zhou}\\[.01in]
	Department of Mathematics and Statistics\\
	Concordia University, Montreal, Canada H3G 1M8\\
	E-mail: xiaowen.zhou@concordia.ca\\[.01in]
}

\maketitle

\begin{abstract}
In this paper, we consider a class of generalized continuous-state branching processes obtained by Lamperti type time changes of spectrally positive L\'evy processes using different rate functions.
When explosion occurs to such a process, we show that the process converges to infinity in finite time asymptotically along a deterministic curve, and identify the speed  of explosion for rate function in different regimes.
To prove the main theorems, we also establish a new asymptotic result  for  scale function of spectrally positive L\'evy process.
\end{abstract}



\section{Introduction}
A continuous-state branching process is a nonnegative real-valued Markov process satisfying the additive branching property.
It arises as time-space scaling limit of discrete Bienaym\'e-Galton-Watson processes.
On the other hand, it can also be obtained by the Lamperti time change of a spectrally positive L\'evy process stopped at hitting $0$ for the first time. 
We refer to \cite{Li_Z2012} and Chapter 12 of \cite{Kyprianou2014:book:levy} for nice introductions on continuous-state branching processes.

The classical Bienaym\'e-Galton-Watson branching processes had been generalized to those with nonadditive branching mechanism; see for example, \cite{SeZu74}, \cite{Kle84}, \cite{Chen02} and \cite{Chen08}. In the same spirit,
continuous-state branching processes with nonadditive branching have been proposed in recent years. In particular, the continuous-state polynomial branching process is introduced in \cite{LiPS2018} as the unique nonnegative solution to a generalized version of the stochastic differential equation in \cite{DaLi06},
which can be identified as a continuous-state branching process with nonadditive, population dependent branching mechanism. The behaviors of extinction, explosion and coming down from infinity for such a process are discussed in \cite{LiPS2018}. A more general class of continuous-state branching processes is proposed in \cite{Foucart2019} via Lamperti type time change of stopped spectrally positive L\'evy processes using rate functions $R$ defined on $(0, \infty)$, where the classical continuous-state branching process corresponds to the linear rate function of $R(x)=x$ and the model in \cite{LiPS2018} corresponds to the rate function of $R(x)=x^\theta$.
The above continuous-state nonlinear branching processes are further generalized in \cite{LYZ19} as solutions to more general versions of the Dawson-Li equation.

For the continuous-state nonlinear branching processes, on one hand, the nonadditive branching mechanism allows richer boundary behaviors such as coming down from infinity; on the other hand, many classical techniques based on the additive branching property fail to work.
Criteria for extinction, explosion and coming down from infinity are developed in 
\cite{LiPS2018}, \cite{LYZ19} and \cite{Foucart2019}
%Li (2019), Li et al. (2019) and Foucart et al. (2019) 
for the respective continuous-state nonlinear branching processes via a martingale approach and fluctuation theory for spectrally positive L\'evy processes.

The speed of coming down from infinity for such processes is studied in \cite{Foucart2019} by analyzing the asymptotic behaviors of weighted occupation times for the associated spectrally positive L\'evy process. Sufficient conditions are found under which the continuous-state nonlinear branching process comes down from infinity along a deterministic curve.

For the continuous-state nonlinear branching processes introduced in \cite{Foucart2019}, explosion occurs when the process $X$ has a positive drift and the rate function increases fast enough near infinity.
In this paper we study the explosion behaviors for such a continuous-state branching process $X$. In particular, we identify the speed of explosion that is defined as the asymptotic of $X(T^+_\infty-t)$ as $t\rightarrow 0+$ for the explosion time $T^+_\infty$. We are not aware of any previous results on the speed of explosion for general Markov processes or for solutions to general stochastic differential equations with jumps. In addition, when explosion happens, using techniques from \cite{LP2018} we also express the potential measure of the process $X$ using the generalized scale functions for the associated spectrally negative L\'evy process.

To find the speed of explosion, we treat separately two classes of rate functions, the so called slow regime of rate functions that are perturbations of power functions and the fast regime of rate functions that are perturbations of exponential functions.
Our approach relies on analyzing the weighted occupation time for spectrally positive L\'evy process.
 For the process $X$ with rate function from the slow regime, given the explosion occurs we can show that the normalization of random variable $T^+_\infty-T^+_x$ converges to $1$ in the conditional probability, where $T^+_x$ denotes the first upcrossing time of level $x$. Similarly, if the rate function belongs to the fast regime, under the conditional probability of explosion the random variable $T^+_\infty-T^+_x$, after rescaling, converges in distribution to a random variable whose distribution can be specified using functionals of spectrally positive L\'evy process. The convergence results in both cases lead to an asymptotics on the running maximum of the process near the explosion time. By comparing values of the associated spectrally positive L\'evy process with its running maximum, we can
show that for rate functions in both regimes the explosion occurs in an asymptotically deterministic fashion. In particular, in the fast regime the speed of explosion is asymptotically proportional to $-\log t$ as time $t\rightarrow 0+$.

Some parts of our approach resemble those in \cite{Foucart2019} and in \cite{Bansaye16} for studying the coming down from infinity behaviors of the respective processes.
But an additional difficulty emerges in our work due to the overshoot when the nonlinear branching process first upcrosses a level $x$ at time $T^+_x$. 
We remark that this difficulty seems to be essential and we do not see schemes such as time reversal can easily get around the problem caused by overshoot.

To overcome this difficulty, for the associated spectrally positive L\'evy process we identify the Laplace transform of its stationary overshoot distribution, and we obtain a new asymptotic result on the corresponding scale function. For the case of fast regime, instead of showing the convergence of Laplace transform for the weighted occupation time as in \cite{Foucart2019}, we apply the occupation density theorem to the weighted occupation time and the properties of regularly varying functions to show the almost sure convergence that eventually leads to the desired convergence in law.

We also want to point out that our condition on rate function, \asptwo\ in Section 3, for the convergences of the rescaled explosion time  can be more general than those in \cite{Foucart2019}.

The rest of the paper is arranged as follows. In Section 2 we first introduce some preliminary results on spectrally positive L\'evy processes and the associated scale functions together with the exit problems and the weighted occupation times. The continuous-state nonlinear branching processes are also defined via the Lamperti type transforms in this section. The main results are presented in Section 3.  All the proofs are deferred to Section 4. Several intermediate results are also posed and proved in Section 4.

%\newpage

\section{Spectrally positive L\'evy processes and continuous-state nonlinear branching processes}

%%%%%%%%%%%%%%%%
%%%% definition
%%%%%%%%%%%%%%%%
Let $\xi$ be a spectrally positive L\'evy process (SPLP), that is a real-valued stochastic process with stationary independent increments and with no negative jumps, defined on a filtered probability space $(\Omega,\mathscr{F}, (\mathscr{F}_{t})_{t\geq0}, \bP)$.
Its Laplace exponent is well-defined and of the L\'{e}vy-Khintchine form,
i.e. for $s\geq0$,
\begin{equation*}
\begin{split}
\D \psi(s)
&:=t^{-1}\log\bE\big(\exp(-s \xi_{t})\big)\\
&=\frac{\sigma^{2}}{2} s^{2}- \mu s
 +\int_{(0, \infty)}\big(e^{-s x}-1+ s x\mathbf{1}(x<1) \big) \Pi (d x),
\end{split}
\end{equation*}
where $\mu\in\mathbb{R}$, $\sigma\geq 0$ and the L\'{e}vy measure $\Pi$ is a $\sigma$-finite measure on $(0,\infty)$ such that $\int_{0}^{\infty} (1\wedge x^{2}) \Pi (dx)<\infty$.
It is well-known that $\psi(\cdot)$ is continuous and strictly convex on $[0,\infty)$,
its right inverse is defined by $\Phi(t):=\sup\{s\geq0, \psi(s)=t\}$.
%We use the standard notations $\bP_{t}f(x):=\bE_{x}\big(f(\xi_{t})\big)$.

Denote by $\bP_{x}$ the probability law of $\xi$ for $\xi_{0}=x$,
and write $\bP$ when $\xi_{0}=0$.
We denote throughout this paper
\begin{equation}\label{eqn:defn:c}
p:=\Phi(0)\quad\text{and}\quad
\gamma:=\bE(\xi_{1})=-\psi'(0).
\end{equation}
Notice that $\gamma<\infty$ if and only if $\int_{0}^{\infty}(x\wedge x^{2})\Pi(dx)<\infty$, and $\gamma=\int_{1}^{\infty}x\Pi(dx)+\mu$.
If $p>0$, then $\psi'(0)\in(0,\infty]$, $\xi$ is transient and goes to $\infty$ as $t\rightarrow\infty$, and the following result holds.
%% lem 5
%% lem 5
%% lem 5
\begin{lem}\label{lem:5}
Write $\D\bar{\xi}_{t}:=\sup_{0\leq s\leq t} \xi_s$ for the running maximum of $\xi$. 
If $p>0$, we have
\[\xi_{t}/\bar{\xi}_{t}\underset{t\to\infty}{\longrightarrow} 1
\quad\text{$\bP$-a.s. and}\quad
\inf_{s>t}{\xi_{s}}/{\xi_{t}}\underset{t\to\infty}{\longrightarrow} 1
\quad\text{in}\,\, \bP.\]
\end{lem}
%%%
\begin{rmk}
If $\psi'(0)=0$, then $\xi$ oscillates and $\D \limsup_{t\to\infty} \xi_{t}/\bar{\xi}_{t}=1$\,\, $\bP$-a.s.
\end{rmk}

%%%%%%%%%%%%%%%%
%%%% scale function
%%%%%%%%%%%%%%%%
For $q\geq 0$, the $q$-scale function $W^{(q)}$
is a continuous and increasing function on $[0,\infty)$
with $W^{(q)}(x)=0$ for $x<0$,
which satisfies
\[
\int_{0}^{\infty} e^{-sy} W^{(q)}(y)\,dy=\frac{1}{\psi(s)-q} \quad\text{for $s>\Phi(q)$}.
\]
We write $W(x)=W^{(0)}(x)$ when $q=0$.
Define the first passage times of $\xi$ as
\[
\tau_{x}^{-}:=\inf\{t>0, \xi_{t}<x\}
\quad\text{and}\quad
\tau_{x}^{+}:=\inf\{t>0, \xi_{t}>x\}
\]
with the convention $\inf\emptyset=\infty$.
Given the scale function,
the following first passage results can be found in section 8.1 and 8.2 of \cite{Kyprianou2014:book:levy},
for $q\geq0$ and $c<x<b$
\begin{equation}\label{exit_Lap}
\bE_{x}\big(e^{-q\tau_{c}^{-}}; \tau_{c}^{-}<\tau_{b}^{-}\big)
= \frac{W^{(q)}(b-x)}{W^{(q)}(b-c)}
\quad\text{and}\quad
\bE_{x}\big(e^{-q\tau_{c}^{-}}\big)=e^{-\Phi(q)(x-c)}
\end{equation}
with the convention $e^{-\infty}=0$.
The potential measure of $\xi$ killed upon leaving interval $[0,\infty)$ is given by
\begin{equation}\label{eqn:resolvent}
\begin{array}{rl}
U(x,dy):=&\D \int_{0}^{\infty}\bP_{x}(\xi_{t}\in dy; t<\tau_{0}^{-}) dt\\
=&\D \big(e^{-p x}W(y)-W(y-x)\big) dy
=:u(x,y)\,dy
\quad\text{for $x,y>0$}.
\end{array}
\end{equation}

%%%%%%%%%%%%%%%%
%%%% A change of measure
%%%%%%%%%%%%%%%%
Change of measure is another useful tool for the fluctuation theory of L\'evy processes.
For $\alpha\in\mathbb{R}$ with $\psi(\alpha)<\infty$, the process $(e^{-\alpha\xi_{t}- \psi(\alpha) t})_{t\geq0}$ is a martingale under $\bP$.
Define the probability measure $\bP^{(\alpha)}$ by
\[
\left.\frac{d \bP^{(\alpha)}}{d \bP} \right|_{\mathscr{F}_t}= e^{-\alpha \xi_t-\psi(\alpha) t} \quad\text{for $t>0$}.
\]
It is well-known that $\xi$ is still a SPLP under $\bP^{(\alpha)}$. The associated Laplace exponent and scale functions under $\bP^{(\alpha)}$ are denoted similarly with subscript $\alpha$.
A direct calculation shows that
\[
\psi_{\alpha}(s)= \psi(\alpha+ s)- \psi(\alpha)\quad\text{and}\quad
\Phi_{\alpha}(s)=\Phi(\psi(\alpha)+s)- \alpha\quad \text{for}\quad s\geq0,
\]
and $W_{\alpha}^{(q)}(x)= e^{-\alpha x} W^{(q+\psi(\alpha))}(x)$.
In particular,
\begin{equation}\label{scale_limit}
W_{p}(y)=e^{-py}W(y)\uparrow W_{p}(\infty)=\Phi'(0).
\end{equation}
We refer to \cite{Kuznetsov2012:scalefunction} and \cite{Hubalek2011:scalefunction:examples} for a more detailed discussions and examples of scale functions.

The following limiting result on the resolvent density in \eqref{eqn:resolvent} is useful in this paper, and we refer to Theorem I.21 of \cite{Bertoin96:book} for a similar result called ``renewal theorem''.
%lem 1
\begin{lem} \label{lem:1}
If $p,\gamma\in(0,\infty)$, we have for any $k\in\mathbb{R}$,
\[
\lim_{y\to\infty}(e^{-px}W(x+y)-W(y))=\frac{1-e^{-px}}{\gamma}
\]
uniformly for all $x\in [k,\infty)$. Therefore,
\[\lim_{y,x\to\infty} \big(e^{-px}W(x+y)-W(y)\big)=1/\gamma.\]
\end{lem}
%rmk
\begin{rmk}
By change of measure, we obtain the following general result
where a light-tailed condition on $\Pi$ is required.
For $q\geq0$, let $\phi(q)$ be the left-root of $t\to \psi(t)-q$ and $\psi'(\phi(q))\in(-\infty,0)$, 
that is that is $\phi(q)<\Phi(q)$ with $\psi(\phi(q))=q$, then $\phi(q)\leq 0$ and
\[
e^{-\phi(q) y}\big(e^{-\Phi(q)x}W^{(q)}(x+y)-W^{(q)}(y)\big)
\underset{y\to\infty}{\longrightarrow}
\frac{e^{(\phi(q)-\Phi(q))x}-1}{\psi'(\phi(q))}
\]
where $-\phi(q)$ is also known as the unique nonnegative root of the Cram\'er-Lundberg equation $\psi(-t)=q$ in risk theory.
\end{rmk}

%%%%%%%%%%%%%%%%
%%%% overshoot
%%%%%%%%%%%%%%%%
The proof is based on the following result used in \cite{Doring2015:integral},
see also Theorem 5.7 of \cite{Kyprianou2014:book:levy} and \cite{BERTOIN199965}.
If $\gamma\in(0,\infty)$ then
\begin{equation}
\bP\big(\xi(\tau_{y}^{+})-y\in dz\big)
\underset{y\to\infty}{\Longrightarrow}
 \rho(dz)
\end{equation}
for some non-degenerate weak limit $\rho$ on $[0,\infty)$,
called \textit{the stationary overshoot distribution} in \cite{Doring2015:integral},
which is characterized in the following lemma, see also Lemma 3 of \cite{BERTOIN2011} for L\'evy process in a half-line.
%%% lemma 6
\begin{lem}\label{lem:6}
If $p,\gamma\in(0,\infty)$, we have for $s\geq0$
\begin{equation}
\widehat{\rho}(s):=
\int_{0-}^{\infty}e^{-sz}\rho(dz)
=\frac{p\psi(s)}{\gamma s(s-p)}.\label{eqn:defn:rho}
\end{equation}
In particular, $\widehat{\rho}(p)=(\gamma \Phi'(0))^{-1}$.
\end{lem}

\

%%%%%%%%%%%%%%%%%%%%%\\\
%%%%%%\subsection*{Continuous-state nonlinear Branching process}
%%%%%% omega function %%%%%%\\\
%%%%%%%%%%%%%%%%%%%%%\\\
The continuous-state nonlinear Branching process $X$ considered in this paper is defined in \cite{LiPS2018} by time changing a spectrally positive L\'evy process.
More precisely, for a function $R(\cdot)$ on $(0,\infty)$, which is positive and locally bounded away from $0$, define an additive functional
\begin{equation}\label{eqn:integral}
\eta(t):=\int_{0}^{t}\frac{1}{R(\xi_{s})}\,ds
\quad\text{for $t<\tau_{0}^{-}$},
\end{equation}
and $\eta(\infty):=\lim_{t\to\infty}\eta(t)$ on the event $\{\tau_{0}^{-}=\infty\}$. 
Its right inverse function is defined as $\D \eta^{-1}(t):=\inf\{s>0, \eta(s)>t\}$ for $t<\eta(\tau_{0}^{-})$.
Then the process $X$ is defined, stopped at time $\eta(\tau_{0}^{-})\leq\infty$, by letting $\D X_{t}:=\xi(\eta^{-1}(t))$ for $t\in [0,\eta(\tau_{0}^{-}))$. It is true that $X$ is a well-defined positive-valued Markov process with absorbing states $\{0,\infty\}$.

Define the first passage times of $X$ by
\[
T_{x}^{-}:=\inf\{t>0, X_{t}<x\},
\quad
T_{x}^{+}:=\inf\{t>0, X_{t}>x\},
\]
for $x\in(0,\infty)$ and
\[\D T_{0}^{-}=\inf\{t>0, X_{t}=0\},\quad T_{\infty}^{+}=\inf\{t>0, X_{t}=\infty\},\]
with the convention $\inf\emptyset=\infty$.
 The following identities on the first passage times follow immediately from the Lamperti type transform. For any $x>0$ we have
 \begin{equation}\label{eqn:passagetimes}
T_{x}^{+}=\eta(\tau_{x}^{+})
\text{\ on the event $\{\tau_{x}^{+}<\tau_{0}^{-}\}$ and \ }
T_{x}^{-}=\eta(\tau_{x}^{-})
\text{\ on the event $\{\tau_{x}^{-}<\infty\}$}.
\end{equation}
In addition, for the absorbing time $\eta(\tau_{0}^{-})$, we have 
\[\eta(\tau_{0}^{-})
=\bigg\{\begin{array}{l@{\text{\quad on the event\quad }}l}
T_{0}^{-}=\eta(\tau_{0}^{-}) & \{\tau_{0}^{-}<\infty\},\\
T_{\infty}^{+}=\eta(\infty) & \{\tau_{0}^{-}=\infty\}.
\end{array}\]

More precisely, at $\eta(\tau_{0}^{-})$, 
the process $X$ \textit{becomes extinct} at the finite time $T_{0}^{-}=\eta(\tau_{0}^{-})$ with $X(T_{0}^{-})=0$ on the event $\{\eta(\tau_{0}^{-})<\infty, \tau_{0}^{-}<\infty\}$;
it \textit{extinguishes} when $\D\lim_{t\to\infty}X(t)=0$ on the event $\{\eta(\tau_{0}^{-})=\infty, \tau_{0}^{-}<\infty\}$;
it \textit{explodes} at the finite time $T_{\infty}^{+}=\eta(\infty)$ with $X(T_{\infty}^{+})=\infty$ on the event $\{\eta(\infty)<\infty,\tau_{0}^{-}=\infty\}$;
and it \textit{drifts to infinity} when $\D\lim_{t\to\infty}X(t)=\infty$ on the event $\{\eta(\infty)=\infty,\tau_{0}^{-}=\infty\}$.
 $\D T_{0}^{-}$ is called the \textit{extinction time} of $X$ if $\D T_{0}^{-}<\infty$, and $\D T_{\infty}^{+}$ is called the \textit{explosion time} of $X$ if $\D T_{\infty}^{+}<\infty$.

We first characterize the extinction and explosion conditions for the process $X$ using integral tests. Note that similar results are obtained in \cite{LiPS2018} for power function $R$.

%%%%%%%%%%%%%%%
%%%%% cor1. %%%%%
%%%%%%%%%%%%%%%
\begin{prop}\label{cor:1}
Extinction occurs for the process $X$ with a positive probability, that is
\[
\bP_{x}(\eta(\tau_{0}^{-})<\infty, \tau_{0}^{-}<\infty)>0
\quad\text{if and only if}\quad
\int_{0+}\frac{W(z)}{R(z)}\,dz<\infty
\]
Moreover, in this case, $\bP_{x}(\eta(\tau_{0}^{-})<\infty|\tau_{0}^{-}<\infty)=1$ for all $x>0$.

If $p,\gamma\in(0,\infty)$, the process $X$ explodes with a positive probability, that is
\[
\bP_{x}(\eta(\infty)<\infty, \tau_{0}^{-}=\infty)>0
\quad\text{if and only if}\quad
\int^{\infty}{\frac{1}{R(z)}}\,dz<\infty
\]
Moreover, in this case, $\bP_{x}(\eta(\infty)<\infty|\tau_{0}^{-}=\infty)=1$ for all $x>0$.
\end{prop}
%%%%%%%%%%%%
%%%%% explosion condition
%%%%%%%%%%%%

We remark here that by the first passage identities \eqref{eqn:passagetimes}, 
the event of explosion is equivalent to the finiteness of the so-called perpetual integrals of 
spectrally negative L\'evy processes on the set $\{\tau_{0}^{-}=\infty\}$,
which has been studied under different conditions; see for example \cite{Doring2015:integral, LPZ:integral, kolb2020} and the references therein.

In the paper, we first introduce the following assumptions on $R$, 
%\[\text{\hypertarget{asp:0}{$\mathbb{H}_{0}$:}}
%\quad\int^{\infty}\frac{dy}{R(y)}<\infty,\]
\begin{center}
\hypertarget{asp:0}{$\mathbb{H}_{0}$:}
$\D \int^{\infty}\frac{dy}{R(y)}<\infty$
\quad\text{and}\quad
\hypertarget{asp:1}{$\mathbb{H}_{1}$:} $\D \int_{0+}^{\infty}\frac{W_{p}(y)}{R(y)}\,dy<\infty$,
\end{center}
and denote by, under the \textit{explosion condition} \aspzero,
\begin{equation}\label{def_phi}
\varphi(x):=\frac{1}{\gamma}\int_{x}^{\infty}\frac{dy}{R(y)}
\quad\text{for $x>0$.}
\end{equation}

\begin{rmk}	
	For general rate function $R$, we may have
$\bP\big(\{T_{\infty}^{+}=\infty\}\cap\{T_{0}^{-}=\infty\}\big)>0$.
Note that the associated process $\xi$ either reaches $0$ or goes to $\infty$ before reaching $0$. Then in this case  with a positive probability
 $X$ either drifts to $\infty$ or extinguishes.
Proposition \ref{cor:1} shows that,
under condition \aspzero\ we have
$\{T_{\infty}^{+}<\infty\}=\{\tau_{0}^{-}=\infty\}$,
and under additional condition \aspone, we have
$\{T_{0}^{-}<\infty\}=\{\tau_{0}^{-}<\infty\}$, which gives
$\bP\big(\{T_{\infty}^{+}<\infty\}\cup\{T_{0}^{-}<\infty\}\big)=1$.
\end{rmk}

\begin{rmk}
If $p>0$, then $W_{p}(\infty)<\infty$ and the condition \aspone\ is equivalent to
\[\D \int_{0+}^{1}\frac{W(y)}{R(y)}\,dy+ \int_{1}^{\infty}\frac{1}{R(y)}\,dy<\infty,\]
which, although stronger than the explosion condition \aspzero,
allows to find explicit expressions for general $R$ for further analysis; c.f. Corollary \ref{prop:3},  Remarks \ref{rmk:4} and \ref{rmk:3}.
\end{rmk}

%In this paper for the main results on explosion $\{T_{\infty}^{+}<\infty\}$,
%$R$ is assumed to satisfy the \textit{explosion condition} \aspzero:
%\[\text{\hypertarget{asp:0}{$\mathbb{H}_{0}$:}}
%\quad\int^{\infty}\frac{dy}{R(y)}<\infty,\]
%and denote 
%%\begin{equation}\label{def_phi}
%%\varphi(x):=\frac{1}{\gamma}\int_{x}^{\infty}\frac{dy}{R(y)}
%%\quad\text{for $x>0$.}
%%\end{equation}
%For results on extinction $\{T_{0}^{-}<\infty\}$, we further assume the condition
%\begin{center}
%\hypertarget{asp:1}{$\mathbb{H}_{1}$:} $\D \int_{0+}^{\infty}\frac{W_{p}(y)}{R(y)}\,dy<\infty$.
%\end{center}

\begin{rmk}
By Proposition VII.10 in \cite{Bertoin96:book}, $W_{p}(x)\asymp \frac{1}{x\psi_{p}(1/x)}$, we have 
%\[
%\begin{gathered}
%\int_{0+}\frac{W(z)}{R(z)}dz<\infty
%\quad\text{if and only if}\quad
%\int_{0+}\frac{1}{z{\color{red}\psi}(1/z)R(z)}dz {\color{red} \quad\text{if and only if}\quad
%	\int^\infty \frac{1}{{\color{red}z\psi}(z)R(1/z)}dz},\\
%\int^{\infty}\frac{W_{p}(z)}{R(z)}dz<\infty
%\quad\text{if and only if}\quad
%\int^{\infty}\frac{1}{z\psi_{p}(1/z)R(z)}dz.
%\end{gathered}
%\]
\[
\begin{gathered}
\int_{0+}\frac{W(z)}{R(z)}dz<\infty
\quad\text{if and only if}\quad
%\int_{0+}\frac{1}{z{\color{red}\psi}(1/z)R(z)}dz {\color{red} \quad\text{if and only if}\quad
%	\int^\infty \frac{1}{{\color{red}z\psi}(z)R(1/z)}dz},\\
\int^{\infty}\frac{1}{z\psi_{p}(z)R(1/z)}dz<\infty\\
\text{and}\\
\int^{\infty}\frac{W_{p}(z)}{R(z)}dz<\infty
\quad\text{if and only if}\quad
\int_{0+}\frac{1}{z\psi_{p}(z)R(1/z)}dz<\infty.
\end{gathered}
\]

In particular, for  $R(z)=z^{\theta}$, by change of variable, we have from Proposition \ref{cor:1}
%for all $b>a>0$
%\[
%\int_{a}^{b}\frac{1}{z\psi_{p}(1/z)R(z)}dz
%=\int_{a}^{b}\frac{z^{-1-\theta}}{\psi_{p}(1/z)}dz
%=\int_{1/b}^{1/a} \frac{z^{\theta-1}}{\psi_{p}(z)}dz.
%\]
%We have
\[
\bP_{x}\big(T_{0}^{-}<\infty\big)>0
\quad\text{if and only if}\quad
\int^{\infty}\frac{z^{\theta-1}}{\psi_{p}(z)}dz<\infty,
\]
which coincides with Theorem 1.8 of \cite{LiPS2018}.
If $p,\gamma\in(0,\infty)$, then 
\[
\bP_{x}\big(T_{\infty}^{+}<\infty\big)>0
\quad\text{if and only if}\quad \theta>1
\]
which coincides with Theorem 1.10(1) of \cite{LiPS2018}.
\end{rmk}

%\cite{kolb2020}

%%%%%%%%%%%%%%%%%%
%%%%%% {Main results}
%%%%%%%%%%%%%%%%%%
\section{Main results}
Under the condition \aspone, the explosion time $T_{\infty}^{+}$ has finite exponential moment 
\begin{thm}\label{thm:4}
Assume $p>0$ and the condition \aspone\ holds for function $R$.
Then $m_{n}(x):= \bE_{x}\big((T_{\infty}^{+})^{n}; T_{\infty}^{+}<T_{0}^{-}\big)$
is finite and can be obtained recursively by
\begin{equation}\label{eqn:defn:rn}
m_{n}(x)=n\int_{0}^{\infty}u(x,y)\omega(y)m_{n-1}(y)\,dy
\quad\text{with}\quad
m_{0}(x)=1-e^{-px}.
\end{equation}
For $\D |q|<\Big(\int_{0+}^{\infty}\frac{W_{p}(z)}{R(z)}\,dz\Big)^{-1}$, we have for $x>0$
\[
\bE_{x}\big(e^{q\cdot T^{+}_{\infty}}; T_{\infty}^{+}<T_{0}^{-}\big)
=\sum_{n=0}^{\infty} \frac{q^{n}}{n!} m_{n}(x)<\infty.
\]
%Suppose that the condition \aspone\ holds for function $R$. Assuming $p>0$, we have
%for $\D |q|<\Big(\int_{0+}^{\infty}\omega(z)W_{p}(z)\,dz\Big)^{-1}$,
%where $m_{n}(x):= \bE_{x}\big((T_{\infty}^{+})^{n}; T_{\infty}^{+}<T_{0}^{-}\big)$
%%denotes the $n$-th moment of $T_{\infty}^{+}$ on the set $\{ T_{\infty}^{+}<T_{0}^{-}\}$ which
%is finite and can be obtained recursively by
%\begin{equation}\label{eqn:defn:rn}
%m_{n}(x)=n\int_{0}^{\infty}u(x,y)\omega(y)m_{n-1}(y)\,dy
%\quad\text{with}\quad
%m_{0}(x)=1-e^{-px}.
%\end{equation}
\end{thm}

Recall $p,\gamma$ defined in \eqref{eqn:defn:c} and $\varphi$ defined in \eqref{def_phi}.
To study the asymptotic behaviors of the process $X$ near time $T_{\infty}^{+}$ on $\{T_{\infty}^{+}<\infty\}$, 
we always assume $p,\gamma\in(0,\infty)$, the explosion condition \aspzero\ and the following condition hold,  for some $\lambda\in [0,\infty)$ and all $x>0$
\begin{center}
\hypertarget{asp:2}{$\mathbb{H}_{2}$:} $\D
\frac{\varphi(x+y)}{\varphi(y)}\underset{y\to\infty}{\longrightarrow} e^{-\lambda x}$.
\end{center}
Denote by 
\[
\begin{gathered}
\bq_{t}(x,A):= \bP_{x}(X_{t}\in A)=\bP_{x}(X_{t}\in A, t<T_{0}^{-}\wedge T_{\infty}^{+}),\\
\bq^{\uparrow}_{x}(B):= \bq_{x}\big(B\big|T_{\infty}^{+}<\infty\big)=\bP_{x}\big(B\big|\tau_{0}^{-}=\infty\big)
=:\bP_{x}^{\uparrow}\big(B\big),
\end{gathered}
\]
for $x>0, t\geq0, A\in\mathscr{B}(0,\infty)$ and $B \in \sigma\{{X}_{t}, t\geq0\}\subset \sigma\{{\xi}_{s}, s\geq0\}$. Then $\bq_{t}$ defines the semigroup of $X$ before absorption,
$\bq^{\uparrow}_{x}$ defines the probability law of $X$ conditioned on explosion,
and $\bP_{x}^{\uparrow}$ denotes the probability law of $\xi$ conditioned to stay positive. 

\begin{rmk}\label{rmk:5}
	Recall that, a function $f>0$ defined on $(0,\infty)$ is
	\textit{regularly varying with index $\alpha\in\mathbb{R}$ at $\infty$} if for any $s>0$,
	\[
	f(s x)/f(x)\to s^{\alpha}\quad\text{as}\,\,x\to\infty,
	\]
	and is \textit{slowly varying} at $\infty$ if $\alpha=0$.
	
	The condition \asptwo\ is equivalent to function
	$x\to\varphi(\log x)$ being regularly varying with index $-\lambda\in(-\infty,0]$.
	 If there exists a positive function $f$ such that 
	 $$\D
	\frac{\varphi(x+y)}{\varphi(y)}\underset{y\to\infty}{\longrightarrow} f(x)\quad\text{for all $x>0$},$$
then condition \asptwo\ necessarily holds;
	 see Theorem 1.4.1 of \cite{Bingham1987:book}. Moreover, under condition \asptwo\ we have \[\log\varphi(x+y)-\log\varphi(y)\underset{y\to\infty}{\longrightarrow}-\lambda x.\]
	 It follows from Lemma 1.4.5 of \cite{Bingham1987:book} that \[\varphi(x)=e^{-(\lambda+\epsilon(x))x}\quad \text{
for some function $\epsilon$ satisfying $\epsilon(x)\rightarrow 0 $ as $x\rightarrow \infty$},\] which can also be obtained from a representation of regularly varying function. 
	
	A sufficient condition for the condition \asptwo\ is that function $x\to R(\log x)$ varies regularly with index $\lambda\geq0$, which holds
by applying Karamata's theorem, c.f. Theorem 1.5.11 and Proposition 1.5.9.b of \cite{Bingham1987:book}.
An interesting example for $R$ is a power-like function satisfying the condition \aspzero\ with $\lambda=0$. If
\[\liminf_{x\to\infty} x^{\alpha}R(x)>0
\quad\text{and}\quad
\limsup_{x\to\infty} x^{\beta}R(x)<\infty\]
for some constants $\alpha\geq \beta$ with $\alpha-\beta<1$, 
then \asptwo\ holds with $\lambda=0$.
Actually, under condition \aspzero, we have $\beta<-1$, thus for some constant $c>0$ and $x$ large enough,
\[
\frac{\varphi(x)-\varphi(x+a)}{\varphi(x)}=\frac{\int_{x}^{x+a}\frac{1}{R(y)}\,dy}{\int_{x}^{\infty}\frac{1}{R(y)}\,dy}
\leq c \frac{\int_{x}^{x+a}y^{\alpha}\,dy}{\int_{x}^{\infty}y^{\beta}\,dy}
\leq \frac{-c a}{1+\beta} x^{\alpha-\beta-1}\to 0.
\]
\end{rmk}

We are ready to present our results on explosion
whose proofs are deferred to Section \ref{sec:proof}.
Recall Proposition \ref{cor:1} that, if $p,\gamma\in(0,\infty)$ and \aspzero\ holds, $\bq^{\uparrow}_{x}\big(T_{\infty}^{+}<\infty\big)=1$.

We first present the asymptotic of the residual explosion time after first uncrossing a level.

%convergence results concerning the explosion time.

%%%%%%%%%%%
%%%%% Theorem 1
%%%%%%%%%%%
\begin{thm}\label{thm:1}
Suppose that $p,\gamma\in(0,\infty)$ and $R$ satisfies the conditions \aspzero\ and \asptwo, 
and let $\lambda\geq0$ be the constant in \asptwo.
\begin{enumerate}[(A)]
%1
\item\label{case:thm1:1} If $\lambda=0$, then in $\bq_{1}^{\uparrow}$-probability
\begin{equation}
\frac{T^{+}_{\infty}-T^{+}_{x}}{\varphi(x)}
\rightarrow 1
\quad\text{as $x\to\infty$}.
\end{equation}
%2
\item\label{case:thm1:2} If $\lambda\in(0,\infty)$ and
$\D \limsup_{x\to\infty}
\frac{1}{\varphi^{2}(x)}
\int_{x}^{\infty}\frac{1}{R^{2}(y)}\,dy<\infty$,
then as $x\to\infty$
\begin{equation}
\begin{gathered}
\frac{\varphi(X(T_{x}^{+}))}{\varphi(x)}\Big|_{\bq_{1}^{\uparrow}}
\stackrel{D}{\Longrightarrow}
e^{-\lambda \varrho},
\quad\quad
\frac{T^{+}_{\infty}-T^{+}_{x}}{\varphi(X(T_{x}^{+}))}\Big|_{\bq_{1}^{\uparrow}}
\stackrel{D}{\Longrightarrow}
\lambda\gamma \int_{0}^{\infty} e^{-\lambda \xi_{t}}\,dt\\
\text{and}\quad\quad
\frac{T^{+}_{\infty}-T^{+}_{x}}{\varphi(x)}\Big|_{\bq_{1}^{\uparrow}}
\stackrel{D}{\Longrightarrow}
\lambda\gamma e^{-\lambda \varrho} \int_{0}^{\infty} e^{-\lambda \xi_{t}}\,dt,
\end{gathered}
\end{equation}
where $Z|_{\bq_{1}^{\uparrow}}$ denotes the law of $Z$ under $\bq_{1}^{\uparrow}$,
and where $\varrho$ is a random variable independent of $\xi$
with probability law $\rho$ specified in \eqref{eqn:defn:rho}.
%and $\mathring{\xi}$ is an independent copy of $\xi$ that is also independent of $\varrho$.
\end{enumerate}
\end{thm}

\begin{rmk}
If $x\to R(\log x)$ varies regularly with index $\lambda>0$, it further follows from the Karamata's theorem that
\[\D \int_{x}^{\infty}\frac{R^{2}(x)}{R^{2}(y)}\,dy\to (2\lambda)^{-1}\quad\text{and}\quad \D\int_{x}^{\infty}\frac{R(x)}{R(y)}\,dy\to \lambda^{-1}\quad\text{ as}\quad x\to\infty,\]
thus,
\[\frac{1}{\varphi^{2}(x)}\int_{x}^{\infty}\frac{1}{R^{2}(y)}\,dy\to \frac{\gamma^{2}\lambda}{2}.\]
In particular, if $R(x)e^{-\lambda x}$ varies regularly with index $\alpha$ for some $\lambda\geq0$ and $\alpha\in\mathbb{R}$, then $x\to R(\log x)$ varies regularly with index $\lambda$.
\end{rmk}

We also have the following main result concerning the speed of explosion.
%%%%%%%%%%%
%%%%% Theorem 1
%%%%%%%%%%%
\begin{thm}\label{thm:2}
Suppose that $\gamma\in(0,\infty)$ and $R$ satisfies the conditions \aspzero\ and \asptwo.
\begin{enumerate}[(a)]
\item\label{case:thm2:1}If $\lambda=0$ and
$\D \liminf_{y\to\infty} \frac{\varphi(y)}{\varphi(hy)}\in (1, \infty]$ for every $h>1$,
then we have in $\bq_{1}^{\uparrow}$-probability
\[
\frac{X(T^{+}_{\infty}-t)}{\varphi^{-1}(t)}\to 1
\quad\text{and}\quad
\frac{\inf_{0<s<t}X(T^{+}_{\infty}-s)}{\varphi^{-1}(t)}\to 1
\quad\text{as $t\to0+$},
\]
where $\varphi^{-1}(t):=\sup\{s>0, \varphi(s)>t\}$ is the right inverse of $\varphi$.
\item\label{case:thm2:2} If $\lambda>0$, then we have in $\bq_{1}^{\uparrow}$-probability
\[\frac{X(T^{+}_{\infty}-t)}{-\log t}\to \lambda^{-1}
\quad\text{and}\quad
\frac{\inf_{0<s<t}X(T^{+}_{\infty}-s)}{-\log t}\to \lambda^{-1}
\quad\text{as $t\to0+$}.\]
\end{enumerate}
\end{thm}
\begin{rmk}
If for some $M,m>0$ and $\alpha>1$, $\D m x^{\alpha}<R(x)<M x^{\alpha}$ for all $x$ large enough, then for $h>1$ we have
\[
\frac{\varphi(x)-\varphi(xh)}{\varphi(x)}
=\frac{\int_{x}^{xh}\frac{1}{R(y)}\,dy}{\int_{x}^{\infty}\frac{1}{R(y)}\,dy}
\geq \frac{m\int_{x}^{xh}y^{-\alpha}\,dy}{M\int_{x}^{\infty}y^{-\alpha}\,dy}
=\frac{m}{M} \big(1-h^{1-\alpha}\big)>0,
\]
 and the corresponding function $R$ satisfies condition \eqref{case:thm2:1} in Theorem \ref{thm:2}.
\end{rmk}
\begin{rmk}
For the asymptotic functions in Theorem \ref{thm:1} and \ref{thm:2},
\begin{itemize}
\item if $R(x)=(c+x)^{\theta}$ for $\theta>1$ and any constant $c$, then
\[\D \varphi(x)=\frac{(x+c)^{1-\theta}}{\gamma(\theta-1)}\quad
\text{and}\quad \D\varphi^{-1}(t)\sim (\gamma(\theta-1)t)^{\frac{1}{1-\theta}}\quad
\text{as}\quad t\rightarrow 0+;\]
\item if $R(x)=e^{\lambda x}$ for $\lambda>0$, then
$\D \varphi(x)=(\lambda\gamma)^{-1} e^{-\lambda x}$
and $\D \varphi^{-1}(t)\sim -\lambda^{-1}\log t$\,\, as \,\, $t\rightarrow 0+$.
\end{itemize}
\end{rmk}

\begin{rmk}	
Studying the  explosion behaviors of $X$ for rate function $R$ with arbitrary behavior near $\infty$ seems to be rather challenging since the explosion may allow different speeds when the explosion time is approached in different ways.  To this end, we assume \asptwo\ on the asymptotic behavior of the rate function, which is similar to those assumptions in \cite{Bansaye16} and \cite{Foucart2019}.
\end{rmk}

\begin{rmk}
We assume that  $\gamma\in(0,\infty)$ in both Theorem \ref{thm:1} and Theorem 
\ref{thm:2}. It remains open to identify the speed of explosion for continuous-state nonlinear processes with big jumps in the sense that $\int_1^\infty x\Pi(dx)=\infty $.
\end{rmk}

%\newpage
% \section{Proofs}
% \section{Proofs}
% \section{Proofs}
\section{Proofs}\label{sec:proof}
This section is dedicated to the proofs of the main results.
Lemmas \ref{lem:5}, \ref{lem:1} and \ref{lem:6} for SPLP are of independent interest and are proved first.
They will be applied in the proofs of main results thereafter. Recall $p$ and $\gamma$ defined in \eqref{eqn:defn:c}.

\subsection{Proofs of Lemmas \ref{lem:5}, \ref{lem:1} and \ref{lem:6}}

%%%%%%[Proof of Lemma \ref{lem:1}]
%%%%%%[Proof of Lemma \ref{lem:1}]
%%%%%%[Proof of Lemma \ref{lem:1}]
Our proof of Lemma \ref{lem:5} is based on the It\^{o} excursion theory, where
the compensation formula and the exponential formula for Poisson point process are applied; c.f. Chapter O of \cite{Bertoin96:book}.
Here we use the standard notions in the fluctuation theory of L\'evy process from \cite{Bertoin96:book}.
Let $\chi:=\bar{\xi}-\xi$ be the L\'evy process reflected at its running maximum,
where $\D\bar{\xi}_{t}=\sup_{s\leq t}\xi_{s}$ is the running maximum of $\xi$.
Let $l$ be a local time process of $\chi$ at $0$ and $l^{-1}$ be its right inverse.
Since $\D\lim_{t\to\infty}\xi(t)=\infty$, $\chi$ is a recurrent Markov process.
In addition, $\big(l^{-1}_{s},\xi(l^{-1}_{s})\big)_{s\geq 0}$
defines a proper bivariate subordinator on $(0,\infty)$,
called the ladder process in Chapter VI of \cite{Bertoin96:book},
with a version of its Laplace exponent given by
$\widehat{\kappa}(\alpha,\beta)=\frac{\alpha-\psi(\beta)}{\Phi(\alpha)-\beta}$.
The excursion process $(\epsilon_{s})_{s\geq0}$ of $\chi$ away from $0$, defined by
\begin{equation}
\epsilon_{s}:=\left\{
\begin{array}{c@{\quad}c}
\{\chi(t+l^{-1}_{s-}), 0\leq t<l^{-1}_{s}-l^{-1}_{s-}\}
& \text{if $l^{-1}_{s-}<l^{-1}_{s}$},\\
\Delta & \text{otherwise, }
\end{array}\right.
\end{equation}
for some isolated point $\Delta$, is a Poisson point process with characteristic measure $n$.
Denote by $\bar{\epsilon}$ the associated excursion height process.
%[Proof of Lemma \ref{lem:5}]
%[Proof of Lemma \ref{lem:5}]
%[Proof of Lemma \ref{lem:5}]
\begin{proof}[Proof of Lemma \ref{lem:5}]
Assume $p>0$. We have $\bP(\tau_{1/\varepsilon}^{+}<\infty)=1$ for $\varepsilon\in(0,1)$.
For every $t>\tau_{1/\varepsilon}^{+}$ and $\bar{\xi}(t)-\xi(t)\geq \varepsilon\cdot\bar{\xi}(t)$, 
we have $\bar{\epsilon}_{s}>\varepsilon\cdot\bar{\xi}(t)>1$ and $t\in (l^{-1}(s-), l^{-1}(s))$ where $s=l(t)$ and $\bar{\xi}(t)=\bar{\xi}(l^{-1}_{s-})$. Therefore, by counting the number of those excursions,
\[
\Big\{\varepsilon<\limsup_{t\to\infty} \frac{\bar{\xi}(t)-\xi(t)}{\bar{\xi}(t)}\Big\}
=\Big\{
\#\big\{s>0\big| \bar{\epsilon}_{s}>1, \bar{\epsilon}_{s}>\varepsilon\cdot \bar{\xi}(l^{-1}_{s-})\big\}
:=N_{\varepsilon}=\infty\Big\}.
\]

On the other hand, since $\chi$ is absent of positive jumps,
the law of $\bar{\epsilon}$ given $\bar{\epsilon}>1$ under $n(\cdot)$
is identical to the law of $\D|\inf_{t<\tau_{0}^{+}}{\xi}(t)|$ under $\bP_{-1}$,
that is, for $y>1$
\[n(\bar{\epsilon}>y|\bar{\epsilon}>1)=
\bP_{-1}\Big(\inf_{t<\tau_{0}^{+}}{\xi}(t)<-y\Big)
=\bP_{-1}(\tau_{-y}^{-}<\tau_{0}^{+})=\frac{W(1)}{W(y)}.\]
%{\color{red}where $\underline{\xi}_t:=\inf_{s\leq t}\xi_s $ denotes the running minimum for $\xi$.}

Notice that similar to Lemma VI.2 of \cite{Bertoin96:book}, the heights of the excursion process $\chi$ are independent of
$\big(\bar{\xi}(l^{-1}_{s-}), s>0\big)$.
Therefore, conditioning on $(\bar{\xi}(l^{-1}_{s-}), s>0)$, 
$\D N_{\varepsilon}$ is Poisson distributed with parameter
\begin{equation}
\Big(\int_{0}^{\infty} \frac{n(\bar{\epsilon}>1)W(1)}{W\big(1\vee \varepsilon\cdot\bar{\xi}(l^{-1}_{s-})\big)} ds\Big).
\end{equation}

Since
$W(x)=e^{px}W_{p}(x)\geq W_{p}(1) e^{px}$ for $x\geq1$,
and
$\bar{\xi}(l^{-1}_{s-})=\bar{\xi}(l^{-1}_{s})
=\xi(l^{-1}_{s})$ for almost every $s>0$ $\bP$-almost surely,
by the right-continuity of $\xi$ and the definition of $l^{-1}$,
we have from Fubini's theorem that
\[
\bE\big(\int_{0}^{\infty}
\frac{W_{p}(1)}{W\big(1\vee \varepsilon\cdot\bar{\xi}(l^{-1}_{s-})\big)} ds\Big)\\
\leq \int_{0}^{\infty}\bE\big(e^{-p\varepsilon \cdot \xi(l^{-1}_{s})}\big) ds
=\frac{p(1-\varepsilon)}{-\psi(\varepsilon p)}<\infty,
\]
where for the equality above we use the fact that $\xi(l^{-1}_{s})$ is a subordinator
with Laplace exponent $\frac{-\psi(\beta)}{p-\beta}$.
Therefore, $\bP$ almost surely, we have
\[
\int_{0}^{\infty} \frac{n(\bar{\epsilon}>1)W(1)}{W\big(1\vee \varepsilon\cdot\bar{\xi}(l^{-1}_{s-})\big)} ds
<\infty
\quad \Leftrightarrow\quad
N_{\varepsilon}<\infty
\quad \Leftrightarrow\quad
\limsup_{t\to\infty} \frac{\bar{\xi}(t)-\xi(t)}{\bar{\xi}(t)}\leq \varepsilon,
\]
and finish the proof of the first assertion.

For the second limit,
from the previous result and
the identity $\D \inf_{s>t}\xi_{s}=\xi_{t}+ \inf_{s>t}(\xi_{s}-\xi_{t})$,
it is sufficient to check that $\D \frac{1}{\bar{\xi}_{t}}\inf_{s>t}(\xi_{s}-\xi_{t})$ converge to $0$ in probability. From the Markov property and the fact that $\bar{\xi}_{t}\to\infty$ as $t\to\infty$, the desired conclusion follows.
\end{proof}

%%\begin{proof}[Proof of Lemma \ref{lem:6}]
%%\begin{proof}[Proof of Lemma \ref{lem:6}]
%%\begin{proof}[Proof of Lemma \ref{lem:6}]
\begin{proof}[Proof of Lemma \ref{lem:6}]
The proof is based on the observation that
$\xi$ and its Ladder height process have the same overshoot when first up-crossing a level.
Thus, the stationary overshoot is identical in law to the limit of the overshoot of the ladder process, and where we need the assumption of $\gamma=\bE[\xi_{1}]\in(0,\infty)$.

More specifically, consider a Ladder height process of $\xi$,
which is a subordinator with a version of Laplace exponent
$\widehat{\kappa}(\beta)=\frac{\psi(\beta)}{\beta-p}$,
c.f. Theorem VII.4 of \cite{Bertoin96:book}.
Let $\delta$ and $\nu(dz)$ be the associated drift parameter and jump measure, respectively. Then
we have from Theorem 5.7 of \cite{Kyprianou2014:book:levy} that
\[\begin{split}
\widehat{\rho}(s):=&\ \int_{0-}^{\infty}e^{-sz}\rho(dz)
 = \int_{0}^{\infty} dy\int_{0}^{\infty}\frac{e^{-sz}}{\mu}\nu(dz+y)
 + \frac{\delta}{\mu}\\
 =&\ \frac{1}{\mu}\big(\delta+ \frac{1}{s}\int_{0}^{\infty}(1-e^{-sz})\nu(dz)\big)
 =\frac{\psi(s)}{\mu s (s-p)},\end{split}\]
where $\mu=\widehat{\kappa}'(0)=\frac{\psi'(0)}{-p}\in(0,\infty)$, which finishes the proof.
\end{proof}

\begin{rmk}
From the L\'evy- Khintchine formula, for the case $p,\gamma\in(0,\infty)$,
\[
\widehat{\rho}(s)=\frac{p\sigma^{2}}{2\gamma}+ \frac{p}{\gamma}\int_{0}^{\infty} e^{-sz}
\Big(\int_{0}^{\infty}e^{-\rho y}\bar{\Pi}(y+z)dy\Big)dz.
\]
\end{rmk}

We are now ready to prove Lemma \ref{lem:1} by applying Lemma \ref{lem:6}.
%%%%%%Proof of Lemma \ref{lem:1}]
%%%%%%Proof of Lemma \ref{lem:1}]
%%%%%%Proof of Lemma \ref{lem:1}]
\begin{proof}[Proof of Lemma \ref{lem:1}]
For $y>0$, define the hitting time of $\xi$ by
\[
\tau^{\{y\}}=\inf\{t>0, \xi_{t}=y\}.
\]
Since the process $X$ is absent of negative jumps, then
$\tau^{\{y\}}=\tau_{y}^{+}+ \tau_{y}^{-}\circ\theta_{\tau_{y}^{+}}$, and by \eqref{exit_Lap} and Lemma \ref{lem:6} we have
\begin{equation}\label{eqn:lem1:2}
\bP(\tau^{\{y\}}<\infty)=\bE\big(e^{-p(\xi(\tau_{y}^{+})-y)}\big)
\underset{y\to\infty}{\longrightarrow} \int_{0-}^{\infty}e^{-pz}\rho(dz)
=(\gamma\Phi'(0))^{-1}.
\end{equation}
It is proved in Lemma 3.1 of \cite{LZ:localtime} that
\[
\bP_{x}\big(\tau^{\{y\}}<\tau_{b}^{+}\wedge\tau_{c}^{-}\big)
=\frac{W(b-x)}{W(b-y)}-\frac{W(y-x)}{W(y-c)}\frac{W(b-c)}{W(b-y)}
\]
for $x,y\in(c,b)$. Letting $b\to\infty$, it follows from \eqref{scale_limit} that
\begin{equation}
\bP(\tau^{\{y\}}<\tau_{-x}^{-})=\frac{ e^{-px}W(y+x)-W(y)}{W_{p}(y+x)}
\quad\text{for $x,y>0$}.\label{eqn:lem1:1}
\end{equation}
On the other hand, applying the strong Markov property, we have
\begin{align*}
\bP(\tau^{\{y\}}<\tau_{-x}^{-})
=&\ \bP(\tau^{\{y\}}<\infty)- \bP(\tau_{-x}^{-}<\tau^{\{y\}}<\infty)\\
=&\ \bP(\tau^{\{y\}}<\infty)- \bP(\tau_{-x}^{-}<\tau_{y}^{+})
\bP_{-x}(\tau^{\{y\}}<\infty)\\
=&\ \bE\big(e^{-p(\xi(\tau_{y}^{+})-y)}\big)
- \frac{W(y)}{W(x+y)} \bE\big(e^{-p(\xi(\tau_{x+y}^{+})-(x+y))}\big),
\end{align*}
where we use the fact that $\xi$ is spatially homogenous. Therefore, for $x,y>0$,
\[\frac{e^{-px}W(x+y)-W(y)}{W_{p}(x+y)}
=\bE\big(e^{-p(\xi(\tau_{y}^{+})-y)}\big)
- e^{-px}\frac{W_{p}(y)}{W_{p}(x+y)} \bE\big(e^{-p(\xi(\tau_{x+y}^{+})-(x+y))}\big).\]
Applying \eqref{eqn:lem1:2} the proof is completed.
\end{proof}
% rmk
% rmk
% rmk
\begin{rmk}
From \eqref{eqn:lem1:1} above, letting $x\to\infty$ first and then $y\to\infty$, we have
\[
\bP(\tau^{\{y\}}<\infty)
=\frac{1}{W_{p}(\infty)}e^{py}\big(W_{p}(\infty)-W_{p}(y)\big)
\to \frac{1}{\gamma\Phi'(0)}.
\]
For the case of linear Brownian motion with
$\D \psi(s)=\frac{\sigma^{2}}{2} s^{2}-\mu s$ for some $\sigma,\mu>0$, we have
$\gamma=-\psi'(0)=\mu$, $p=\frac{2\mu}{\sigma^{2}}$
and $\Phi'(0)=(\psi'(p))^{-1}=\mu^{-1}$.
Then $\gamma\Phi'(0)=1$.
On the other hand, due to the absence of jumps and due to the positive drift, we always have $\bP(\tau^{\{y\}}<\infty)=1$ for every $y>0$.
\end{rmk}

\subsection{Proof of Proposition \ref{cor:1}}
Since the processes $X$ and $\xi$ are connected via the Lamperti type time transform, in the proofs of our main results, we focus ourself on the study of $\xi$ and its integral functional.
In the proof, we write $\omega(\cdot):=1/R(\cdot)$. The process $\eta(\cdot)$ in \eqref{eqn:integral} is then written as
\[
\eta(t)=\int_{0}^{t}\frac{1}{R(\xi_{s})}ds=\int_{0}^{t}\omega(\xi_{s})ds
\]
and called the weighted occupation time process in \cite{LP2018}, 
where fluctuation theory of the $\omega$-killed spectrally one-sided L\'evy processes is studied.
In this paper, $\omega$ is always positive and locally bounded on $(0,\infty)$.
Recall that $p$ and $\gamma$ are constants defined in \eqref{eqn:defn:c}, respectively, and the Lamperti type identities between the first passage times for $X$ and $\xi$ in \eqref{eqn:passagetimes}. 

For the proof of Proposition \ref{cor:1}, 
the condition of explosion is an immediate consequence of the following result from \cite{Doring2015:integral}: if $\bE(\xi_{1})\in(0,\infty)$ and $f$ is a positive locally integrable function, then $\D \bP\big(\int_{0}^{\infty}f(\xi_{s})\,ds<\infty\big)\in\{0,1\}$ and
\begin{equation}\label{eqn:ky}
\bP\big(\int_{0}^{\infty}f(\xi_{s})\,ds<\infty\big)=0
\quad\Longleftrightarrow\quad
\int^{\infty}f(x)\,dx=\infty.
\end{equation}
Even for $\gamma=\infty$, one can find from the proof
%\rp{the}{their} proof \del{in \cite{Doring2015:integral}}
for sufficiency in \cite{Doring2015:integral} that the identity on the right hand side of \eqref{eqn:ky} is still a sufficient condition for  the left hand side to hold.
Therefore, we only focus on the proof of  extinction condition in the first statement of Proposition \ref{cor:1},
where we need the following result that extends the classical result of \eqref{exit_Lap} and leave the proof to interested readers; also see  Remark 4 in \cite{LP2018} and Lemma 4.2 of \cite{LZ:localtime}.

% prop6
\begin{prop}\label{prop:6}
%Let $\wo$ be as defined in \eqref{eqn:defn:wo:1}.
 For any $b>x>c>0$, we have
\[\bE_{x}\big(e^{-\eta(\tau_{c}^{-})}; \tau_{c}^{-}<\tau_{b}^{+}\big)=\frac{\wo(b,x)}{\wo(b,c)},\]
where $\wo$ is defined as the unique locally bounded function satisfying
\begin{align}
\wo(x,y)
=&\ W(x-y)+\int_{y}^{x}W(x-z)\omega(z)\wo(z,y)\,dz\label{eqn:defn:wo:1}\\
=&\ W(x-y)+\int_{y}^{x}\wo(x,z)\omega(z)W(z-y)\,dz,\label{eqn:defn:wo:2}
\end{align}
%In addition, the $\omega$-resolvent measure of $\xi$ is given by
%\[\int_{0}^{\infty} \bE_{x}\big(e^{-\eta(t)}; \xi_{t}\in dy, t<\tau_{b}^{+}\wedge\tau_{c}^{-}\big)\,dt
%=\Big(\frac{\wo(b,x)}{\wo(b,c)}\wo(y,c)-\wo(y,x)\Big)\,dy.\]
\end{prop}

%The following $\omega$-scale function $\wo$ is useful and defined as the unique locally bounded function satisfying
%\begin{align}
%\wo(x,y)
%=&\ W(x-y)+\int_{y}^{x}W(x-z)\omega(z)\wo(z,y)\,dz\label{eqn:defn:wo:1}\\
%=&\ W(x-y)+\int_{y}^{x}\wo(x,z)\omega(z)W(z-y)\,dz,\label{eqn:defn:wo:2}
%\end{align}
%for $x,y>0$,
%where the second equation is proved in Lemma 4.2 of \cite{LZ:localtime}.
%The following result that extends the classical result of \eqref{exit_Lap}
%% can be derived from Theorem 2.2 of \cite{LP2018} 
% and we leave the proof to interested readers; see also Remark 4 in \cite{LP2018}.

%{\color{blue}
%To {\color{red} obtain the Laplace transform for} the downward first passage time in Proposition \ref{prop:3}, 
%we only need to let $b\to\infty$, 
%which is sufficient to consider the asymptotic behavior of $W^{(\omega)}(b,x)$.
%We shall see that $\ho$ in \eqref{equ_H} is the desired limit, and Proposition \ref{cor:1} for extinction is proved thereafter.}

For the $\wo(x,y)$ defined above, we have the following asymptotic results
% prop7
\begin{lem}\label{prop:7}
For any $x,y>0$, we have
\begin{equation}
\begin{split}
\wo(x,0+):=\lim_{y\to0+}\wo(x,y)<\infty
&\quad\text{if and only if}\quad
\int_{0+}^{1}\omega(z)W_{p}(z)\,dz<\infty,\\
\text{and}\quad \ho(y):=\lim_{x\to\infty}\frac{\wo(x,y)}{W(x)}<\infty
&\quad\text{if and only if}\quad
\int_{1}^{\infty}\omega(z)W_{p}(z)\,dz<\infty.
\label{eqn:defn:who}
\end{split}
\end{equation}
Moreover, the function $\ho$ defined above satisfies
% \begin{equation}\label{equ_H}
% \ho(y)=e^{-py}+ \int_{y}^{\infty}\ho(z)\frac{W(z-y)}{R(z)}\,dz.
%\end{equation}
%and can be given by 
%\begin{equation}\label{equ_H2}
%\ho(y)=e^{-py}+ \int_{y}^{\infty}e^{-pz}\omega(z)\wo(z,y)\,dz.
%\end{equation}
\begin{equation}\label{equ_H}
\begin{split}
\ho(y)&=e^{-py}+ \int_{y}^{\infty}e^{-pz}\omega(z)\wo(z,y)\,dz\\
&=e^{-py}+ \int_{y}^{\infty}\ho(z)\omega(z)W(z-y)\,dz.
\end{split}
\end{equation}
Therefore, if $\omega$ satisfies \aspone, then
$\D \ho(0+):=\lim_{y\to0+}\ho(y)<\infty$.
\end{lem}
%\begin{proof}[Proof of Proposition \ref{prop:7}]
%\begin{proof}[Proof of Proposition \ref{prop:7}]
%\begin{proof}[Proof of Proposition \ref{prop:7}]
\begin{proof}[Proof of Lemma \ref{prop:7}]
We start from the existence of the limits in \eqref{eqn:defn:who}.

For $x>y>0$, we have from \eqref{eqn:defn:wo:1} that
\begin{align}
\frac{\wo(x,y)}{W(x)}
=&\ \frac{W(x-y)}{W(x)}+ \int_{y}^{x} \frac{W(x-z)}{W(x)}\omega(z)\wo(z,y)\,dz\label{eqn:prop7:2}\\
\leq&\ 1+ \int_{y}^{x}e^{-pz} \omega(z)\wo(z,y)\,dz\label{eqn:prop7:1}.
\end{align}
Put $G(x):=\exp(-\int_{1}^{x}\omega(z)W_{p}(z)\,dz)$.
Then $G$ is absolutely continuous with respect to Lesbegue measure with $G'(x)=-\omega(x)W_{p}(x)G(x)$ for a.e.-$x$, and for a.e.-$x$
\begin{align*}
&\ \frac{\partial}{\partial x}\Big(G(x) \big(1+\int_{y}^{x}e^{-pz} \omega(z)\wo(z,y)\,dz\big)\Big)\\
=&\ G(x)\Big(
e^{-px}\omega(x)\wo(x,y)-\omega(x)W_{p}(x)
\big(1+\int_{y}^{x}e^{-pz} \omega(z)\wo(z,y)\,dz\big)\Big)\\
=&\ G(x)\omega(x) W_{p}(x)
\Big(\frac{\wo(x,y)}{W(x)} -
\big(1+\int_{y}^{x}e^{-pz}\omega(z)\wo(z,y)dz\big)\Big)\leq 0
\end{align*}
by \eqref{eqn:prop7:1}.
Thus, for $x>y>0$
\[ \D G(x)\big(1+ \int_{y}^{x}e^{-pz} \omega(z)\wo(z,y)\,dz\big)\leq G(y).\]
Making use of \eqref{eqn:prop7:1} again gives
\[
\frac{\wo(x,y)}{W(x)}\leq \frac{G(y)}{G(x)}
=\exp\big(\int_{y}^{x}\omega(z)W_{p}(z)\,dz\big)
\quad\text{for $x>y>0$}.
\]

From the inequality above,
if $\D\int_{1}^{\infty}\omega(z)W_{p}(z)\,dz<\infty$, for fixed $y>0$,
 $\D x\to \frac{\wo(x,y)}{W(x)}$ is bounded from above, which,
together with the fact of being increasing in $x$ by \eqref{eqn:prop7:2},
 gives the existence and finiteness of $\ho(y)$ on $(0,\infty)$.
The equations \eqref{equ_H} for $\ho$ follow from \eqref{eqn:defn:wo:1} and \eqref{eqn:defn:wo:2} by applying the monotone convergence theorem.
The ``if'' part in the first assertion on $\wo(x,0+)$ also follows from the inequality above.

On the other hand,
since $\wo(x,y)\geq W(x-y)$, we have from \eqref{eqn:defn:wo:2} that,
\begin{align*}
\wo(x,y)\geq&\ \int_{y}^{x}W(x-z)\omega(z)W(z-y)\,dz\\
\geq&\ e^{p(x-y)}W_{p}(x-c) \int_{y}^{c}\omega(z)W_{p}(z-y)\,dz
\quad\text{for every $c\in(y,x)$.}
\end{align*}
It follows that, $\wo(x,0+)=\infty$ if $\D \int_{0+}^{1}\omega(z)W_{p}(z)\,dz=\infty$.
Moreover, for every $c>y$,
\[
\liminf_{x\to\infty}\frac{\wo(x,y)}{W(x)}\geq e^{-py}\int_{y}^{c}\omega(z)W_{p}(z-y)\,dz,
\]
Thus,
\[\D \liminf_{x\to\infty}\frac{\wo(x,y)}{W(x)}=\infty\quad\text{ if}\quad \D\int_{1}^{\infty}\omega(z)W_{p}(z)\,dz=\infty,\]
which proves the ``only if'' part in the assertions.
$\ho(0+)<\infty$ under \aspone\ also follows. This completes the proof.
\end{proof}

%For the proof of the extinction condition in Proposition \ref{cor:1}. 

%By the Lamperti type transform, we have for every $x>0$
%\begin{gather*}
%T_{x}^{+}=\eta(\tau_{x}^{+})
%\text{\ on the set $\{\tau_{x}^{+}<\tau_{0}^{-}\}$ and \ }
%T_{x}^{-}=\eta(\tau_{x}^{-})
%\text{\ on the set $\{\tau_{x}^{-}<\infty\}$}.
%\end{gather*}
%In addition,
%\[\eta(\tau_{0}^{-})
%=\left\{\begin{array}{l@{\text{\quad on the set\quad }}l}
%T_{0}^{-}=\eta(\tau_{0}^{-}) & \{\tau_{0}^{-}<\infty\},\\
%T_{\infty}^{+}=\eta(\infty) & \{\tau_{0}^{-}=\infty\}.
%\end{array}\right.\]
%Proposition \ref{cor:1} is proved by making use of Proposition \ref{prop:7} on the asymptotic of function $(x,y)\to \wo(x,y)$.
%Recall that $p, \gamma$ are constants defined in \eqref{eqn:defn:c}, respectively, and the Lamperti type identities between the first passage times for $X$ and $\xi$ in \eqref{eqn:passagetimes}.

% Proof of Proposition \ref{cor:1}]
% Proof of Proposition \ref{cor:1}]
% Proof of Proposition \ref{cor:1}]
\begin{proof}[Proof of Proposition \ref{cor:1}(extinction condition)]
%Assume $p,\gamma\in(0,\infty)$.
Letting $c\to0+$ in Proposition \ref{prop:6}, we have
\[
\bE_{x}\big(e^{-\eta(\tau_{0}^{-})}; \tau_{0}^{-}<\tau_{b}^{+}\big)=\lim_{c\to0+}\frac{\wo(b,x)}{\wo(b,c)}=\frac{\wo(b,x)}{\wo(b,0+)},
\]
for every $b>x>0$.
By Lemma \ref{prop:7},
\[\bP_{x}\big(\eta(\tau_{0}^{-})<\infty, \tau_{0}^{-}<\infty\big)>0\quad\text{ if and only if}\quad \D\int_{0+}^{1}\omega(z)W_{p}(z)\,dz<\infty.\]

On the other hand, if $\D\int_{0+}^{1}\omega(z)W_{p}(z)\,dz<\infty$, we also have for every $q>0$,
\[
\bE_{x}\big(e^{-q\cdot\eta(\tau_{0}^{-})}; \tau_{0}^{-}<\tau_{b}^{+}\big)=\frac{W^{(q\omega)}(b,x)}{W^{(q\omega)}(b,0)}>0,
\]
where $W^{(q\omega)}$ is the generalized scale function in \eqref{eqn:defn:wo:1} with respect to $q\omega(\cdot)$. By the scale function identity, for every $x,y,q,r>0$,
\[
W^{(q\omega)}(x,y)-W^{(r\omega)}(x,y)=(q-r)\int_{y}^{x}W^{(q\omega)}(x,z)\omega(z)W^{(r\omega)}(z,y)\,dz,
\]
see Lemma 4.3 of \cite{LZ:localtime}, we have that $q\to W^{(q\omega)}(x,y)$ is increasing.
It is not hard to find that $W^{(q\omega)}(x,y)\to W(x-y)$ as $q\to0+$, which shows that
\[\bP_{x}(\eta(\tau_{0}^{-})<\infty, \tau_{0}^{-}<\tau_{b}^{+})=\bP_{x}\big(\tau_{0}^{-}<\tau_{b}^{+}\big)\]
 and the second assertion is proved.
\end{proof}

Applying Lemma \ref{prop:7} to Proposition \ref{prop:6} by letting $b\to\infty$, 
we also have the the following results on the downward passage time of $X$ and $\ho$,
and we leave the proof to interested readers,
where $p$ defined in \eqref{eqn:defn:c} for the underlying $\xi$ can be $0$ and $\D \lim_{z\to\infty}W_{p}(z)=\Phi'(0)\leq \infty$.
%We  first present the Laplace transform for the downward hitting time of $X$, 
%where $p$ defined in \eqref{eqn:defn:c} for the underlying $\xi$ can be $0$ 
%and $\D \lim_{z\to\infty}W_{p}(z)=\Phi'(0)\leq \infty$.
\begin{cor}\label{prop:3}
Recall $p=\Phi(0)\geq0$ defined in \eqref{eqn:defn:c}.
If $\D\int_{1}^{\infty}\frac{W_{p}(z)}{R(z)}\,dz<\infty$, let $\ho$ be defined in \eqref{eqn:defn:who}.
Then
\[
\bE_{x}\big(e^{-T_{c}^{-}}; T_{c}^{-}<\infty\big)=
\frac{\ho(x)}{\ho(c)},
\]
for every $x>c>0$.
%where $\ho$ is a function on $(0,\infty)$ satisfying 
% \begin{equation}\label{equ_H}
% \ho(y)=e^{-py}+ \int_{y}^{\infty}\ho(z)\frac{W(z-y)}{R(z)}\,dz.
%\end{equation}
If $R$ satisfies \aspone, then 
%$\D \ho(0+):=\lim_{y\to0+}\ho(y)<\infty$ and 
the identity also holds for $c=0$.
\end{cor}

\begin{rmk}\label{rmk:4}
If $\D\int_{1}^{\infty} \frac{W_{p}(z)}{R(z)}dz<\infty$,
we can always express $\ho$ in \eqref{equ_H} in terms of sum of a sequence of integrals.
If however $\D \int_{1}^{\infty}\frac{W_{p}(z)}{R(z)}dz=\infty$, the Laplace transform in Corollary \ref{prop:3} holds for some increasing $\tilde{H}^{(\omega)}$ on $(0,\infty)$ satisfying the singular equation
\[
\tilde{H}^{(\omega)}(x)=\int_{x}^{\infty}\tilde{H}^{(\omega)}(z)\omega(z)W(z-x)dz.
\]

To evaluate $\ho$, if  $\D R(x)=\Big(\int_{0}^{\infty} e^{-xt}\mu(dt)\Big)^{-1}$ 
for some positive measure $\mu$ on $(0,\infty)$ such that the condition \aspone\ is fulfilled, that is, 
\[ \int_{0+}^{\infty}\frac{W_{p}(y)}{R(y)}\,dy
=\int_{0+}^{\infty}\mu(dt)\int_{0}^{\infty}e^{-yt}W_{p}(y)\,dy
=\int_{0}^{\infty}\frac{\mu(dt)}{\psi_{p}(t)}
<\infty,\]
then $\D \ho(y)=\int_{0}^{\infty} e^{-ys}\nu(ds)$ for some positive measure $\nu$ on $[p,\infty)$ such that
\[\nu(ds)=\delta_{\{p\}}(ds)+\frac{\nu*\mu(ds)}{\psi(s)}\quad\text{for $s\geq p$.}\]
If $\mu(ds)=h(s)ds$ for some measurable $h\geq0$ on $(0,\infty)$, 
then $\D \ho(y)=e^{-py}+ \int_{p}^{\infty} e^{-ys}k(s)\,ds$ where $k$ is a locally integrable function on $(p,\infty)$ satisfying a Volterra equation
 \[\D k(s)=\frac{1}{\psi(s)}\big(h(s)+ \int_{p}^{s}h(s-r)k(r)\,dr\big) \text{\quad for\quad} s>p.\]
\end{rmk}

\subsection{Proof of Theorem \ref{thm:4}}
An application of Proposition \ref{cor:1} shows that, under the condition $p,\gamma\in(0,\infty)$ and \aspzero, $\{T_{\infty}^{+}<\infty\}=\{\tau_{0}^{-}=\infty\}$ and has a positive probability, the moment function $m_{n}$ defined in Theorem \ref{thm:4} can now be written in terms of $\xi$ as
\[m_{n}(x)=\bE_{x}\big(\eta^{n}(\infty); \tau_{0}^{-}=\infty\big)\leq \infty\quad\text{for $n\in\mathbb{N}$ and $x\geq0$}.\]
The following proposition \ref{prop:2} on $m_{n}$ is frequently used in our proofs.
%Let
%\[m_{n}(x):=\bE_{x}\big(\eta^{n}(\infty); \tau_{0}^{-}=\infty\big)\leq \infty\quad\text{ for $n\in\mathbb{N}$ and $x\geq0$}.\]
%By Proposition \ref{cor:1} proved above, one can find that under the condition \aspzero, the moment $m_{n}$ defined above coincides with the function defined in Theorem \ref{thm:4}.
%A result similar to the following Proposition \ref{prop:2} 
Similar result can be found in Lemma 8.11.1 of \cite{Bingham1987:book}, and here we provide a proof for readers' convenience.

% \label{prop:2}
% \label{prop:2}
% \label{prop:2}
\begin{prop}\label{prop:2}
Let $U$ be defined in \eqref{eqn:resolvent}.
We have $m_{0}(x)=1-e^{-px}$ and
\[m_{n}(x)=n \int_{0}^{\infty}U(x,dy)\omega(y)m_{n-1}(y)
\quad\text{for $x\geq0$ and $n\geq1$.}\]
\end{prop}
%[Proof of Proposition \ref{prop:2}]
%[Proof of Proposition \ref{prop:2}]
%[Proof of Proposition \ref{prop:2}]
\begin{proof}[Proof of Proposition \ref{prop:2}]
The expression for $m_{0}(x)$ follows from \eqref{exit_Lap} by taking $q=0$.

Since $\tau_{0}^{-}=t+ \tau_{0}^{-}\circ\theta_{t}$ on the set $\{t<\tau_{0}^{-}\}$ for the shifting operator $\theta_{t}$,
we have from $d\eta(t)=\omega(\xi_{t})\,dt$ that, on the set $\{\tau_{0}^{-}=\infty\}$,
\begin{align*}
\ \eta^{n}(\infty)\cdot \mathbf{1}(\tau_{0}^{-}=\infty)
&= n \int_{0}^{\infty}\big(\eta(\infty)-\eta(t)\big)^{n-1}\omega(\xi_{t})\cdot \mathbf{1}(t<\tau_{0}^{-}=\infty)\,dt\\
&= n \int_{0}^{\infty} \big(\eta^{n-1}(\infty)\mathbf{1}(\tau_{0}^{-}=\infty)\big)\circ \theta_{t}\cdot \omega(\xi_{t})\cdot \mathbf{1}(t<\tau_{0}^{-})\,dt.
\end{align*}
By the Markov property at time $t>0$ and Fubini's theorem, we complete the proof.
\end{proof}

%%% rmk
%%% rmk
%%% rmk
\begin{rmk}\label{rmk:3}
If $\gamma\in(0,\infty)$, applying Lemma \ref{lem:1} to Proposition \ref{prop:2} we have
\[
\bE_{x}\big(\eta(\infty); \tau_{0}^{-}=\infty\big)<\infty
\quad\text{if and only if}\quad
\int_{0+}^{\infty} \omega(y)
\big(1\wedge ( y W_{p}(y))\big)\,dy<\infty.
\]
Using the idea similar to Proposition \ref{prop:2} in the following, we have for $x>0$,
\[\bE_{x}\big(\eta(\tau_{0}^{-}); \tau_{0}^{-}<\infty\big)=\int_{0}^{\infty}U(x,dy)\omega(y)\bP_{y}(\tau_{0}^{-}<\infty).\]
Then Lemma \ref{lem:1} shows that if $\gamma\in(0,\infty)$,
\[\begin{split}
\bE_{x}\big(\eta(\tau_{0}^{-}); \tau_{0}^{-}<\infty\big)<\infty
&\quad\text{if and only if}\quad
\int_{0+}^{\infty} \omega(y)\big(e^{-py}\wedge W_{p}(y)\big)\,dy<\infty,\\
\text{and}\quad\bE_{x}\big(\eta(\tau_{0}^{-})\big)<\infty
&\quad\text{if and only if}\quad
\int_{0+}^{\infty} \omega(y)W_{p}(y)\,dy<\infty.
\end{split}\]
We refer to \cite{LPZ:integral} for more detailed discussions on the related results.
Notice that the 0-1 law in the first part of Proposition \ref{cor:1} can also be proved by showing that
$$\bE\big(\eta(\tau_{0}^{-}); \tau_{0}^{-}<\tau_{b}^{+}\big)<\infty
\quad\text{if and only if}\quad 
\int_{0+}^{1}\omega(z)W_{p}(z)\,dz<\infty.$$
\end{rmk}

We are now ready to prove Theorem \ref{thm:4}.
Notice that $\omega$ in Theorem \ref{thm:4} is assumed to satisfy \aspone,
which fulfills the condition of Lemma \ref{prop:7}, and under which
\[\eta(\tau_{0}^{-})=\eta(\tau_{0}^{-})\mathbf{1}(\tau_{0}^{-}<\infty)+\eta(\infty)\mathbf{1}(\tau_{0}^{-}=\infty)<\infty\quad\text{$\bP_{x}$-a.s.}\]
for any $x>0$ as shown in Proposition \ref{cor:1}.

\begin{proof}[Proof of Theorem \ref{thm:4}]
The moment generating function of $T_{\infty}^{+}$ is obtained
from Proposition \ref{prop:2}.
From \eqref{eqn:lem1:1}, we know that the density of $U$ is bounded by
\begin{equation}
u(x,y)=e^{-px}W(y)-W(y-x)\leq W_{p}(y)
\quad \text{for all $x,y>0$}.
\end{equation}
Therefore, with $m_{0}(x)=m_{0}(x)=1-e^{-px}\leq 1$, we have
\[
m_{n}(x)\leq n \int_{0}^{\infty} \omega(y)W_{p}(y) m_{n-1}(y)\,dy
\leq n!\times \Big(\int_{0}^{\infty} \omega(y)W_{p}(y) \Big)^{n}
\quad\text{for all $n\geq1$}.
\]
 Since Carleman's condition on the moments is satisfied, the distribution of
%$T^{+}_{\infty}1_{\{T_{\infty}^{+}<T_{0}^{-}\}}$
$T_{\infty}^{+}=\eta(\infty)$ on the set $\{T_{\infty}^{+}<T_{0}^{-}\}=\{\tau_{0}^{-}=\infty\}=\{T_{0}^{-}=\infty\}=\{T_{\infty}^{+}<\infty\}$ under the condition \aspone\ is uniquely determined by its moments $(m_n)_{n\geq0}$, and the desired conclusion follows. \end{proof}

\subsection{Proofs of Theorems \ref{thm:1} and \ref{thm:2}}
%{\color{red} Make it a Remark???
	
%	Studying the  explosion behaviors of $X$ for rate function $R$ with arbitrary behavior near $\infty$ seems to be rather challenging since the explosion may allow different speeds when the explosion time is approached in different ways.  To this end, we assume \asptwo\ on the asymptotic behavior of the rate function, which is similar to those assumptions in \cite{Bansaye16} and \cite{Foucart2019}.}

%For the explosion behaviors of $X$, the exact analysis is prohibitive.
%We make use of the moment method in the proof.
%Given the randomness of overshoots and the limiting result on the resolvent density in Lemma \ref{lem:1}, 
%the condition \asptwo\ seems to be a proper one for the problem, and it is indeed the case.}

To compare the asymptotic behaviors of functions at infinity,
we write as usual
 \[\begin{split}
f(x)\sim g(x) &\quad\text{ as}\quad x\rightarrow\infty
\quad\text{ if}\quad \D\lim_{x\to\infty}{f(x)}/{g(x)}=1,\\
f(x)=o(g(x)) &\quad\text{ as}\quad x\rightarrow\infty
\quad\text{ if}\quad \D\lim_{x\to\infty}{f(x)}/{g(x)}=0,
 \end{split}\]
where $g(x)\neq 0$ for $x$ large enough. 
We always assume that $p,\gamma\in(0,\infty)$ 
and the weight function $\omega$ satisfies \aspzero\ and \asptwo.
We prove in Proposition \ref{prop:1} the asymptotic results  about the tail integrals for functions of this kind, that is, 
%satisfying both \aspzero\ and \asptwo. 
\begin{equation}\label{neweqn:1}
\int_{x+y}^{\infty}f(z)\,dz\Big/\int_{y}^{\infty}f(z)\,dz\to \exp(-\lambda x)\quad\text{ as}\,\, y\to\infty,
\end{equation}
for some constant $\lambda\in[0,\infty)$ and every $x>0$. 
Since the condition \eqref{neweqn:1} is closely related to regularly varying functions as shown in Remark \ref{rmk:5}, 
similar results for ``Stieltjes-integral forms'' can be found in Theorem 1.6.4 and 1.6.5 of \cite{Bingham1987:book}. Recall the following Karamata's theorem from Theorem 1.5.11 of \cite{Bingham1987:book}.

\begin{prop}[Karamata's Theorem]
Let $f$ vary regularly with index $\rho$, and be locally bounded in $[c,\infty)$. Then 
\begin{enumerate}[(i)]
\item for any $\sigma\geq -(\rho+1)$, 
\[
x^{\sigma+1}f(x)\Big/ \int_{c}^{x}t^{\sigma}f(t)dt\to \sigma+\rho+1 \,\,\text{as}\,\,x\to\infty;
\]
\item for any $\sigma<-(\rho+1)$ (and for $\sigma=-(\rho+1)$ if $\int^{\infty} t^{-(\rho+1)}f(t)dt<\infty)$
\[
x^{\sigma+1}f(x)\Big/ \int_{x}^{\infty}t^{\sigma}f(t)dt\to -(\sigma+\rho+1) \,\,\text{as}\,\,x\to\infty.
\]
\end{enumerate}
\end{prop}
%For a positive function $f$ satisfying the condition \aspzero, its tail integral is defined as $\D\int_{x}^{\infty}f(y)\,dy$,
%and we need the following asymptotic results for its integrals,
%c.f. Theorem 1.6.5 of \cite{Bingham1987:book}.

% \label{prop:1}
% \label{prop:1}
\begin{prop}\label{prop:1}
Suppose that a positive function $f$ has finite tail integral
and its tail integral satisfies the condition \eqref{neweqn:1} for some $\lambda\in[0,\infty)$,
\begin{enumerate}[(A)]
%1
\item\label{a} If $\lambda=0$, then for any $\alpha>0$ we have
\[\D \int_{1}^{\infty}e^{\alpha y}f(y)\,dy=\infty, \quad
\D \int_{1}^{\infty}e^{\alpha y}\int_{y}^{\infty}f(z)\,dz\,dy=\infty\]
\[ \text{and}\quad
e^{-\alpha x}\int_{1}^{x}e^{\alpha y}f(y)\,dy = o\big(\int_{x}^{\infty}f(y)\,dy\big)
\quad \text{as $x\to\infty$},
\]
%2
\item\label{b} If $\lambda>0$, then for any $\alpha<\lambda$ we have
\[\D \int_{1}^{\infty}e^{\alpha y}f(y)\,dy<\infty,\quad
\D \int_{1}^{\infty}e^{\alpha y}\int_{y}^{\infty}f(z)\,dz\,dy<\infty\]
and as $x\to\infty$,
\[\int_{x}^{\infty} e^{\alpha(y-x)}\int_{y}^{\infty}f(z)\,dzdy
\sim
\lambda^{-1} \int_{x}^{\infty}e^{\alpha(y-x)}f(y)\,dy
\sim \frac{1}{\lambda-\alpha} \int_{x}^{\infty}f(y)\,dy.\]
%3
\item\label{c} If $\lambda>0$, denoting by $k$ the inverse of function $x\to\int_{x}^{\infty}f(y)\,dy$, i.e. $\int_{k(x)}^{\infty}f(y)\,dy=x$ for all small $x>0$, we have
\[k(x)\sim -\lambda^{-1}\log x\quad\text{as $x\to0+$}.\]
\end{enumerate}
\end{prop}
%Proof of Proposition \ref{prop:1}
%Proof of Proposition \ref{prop:1}
%Proof of Proposition \ref{prop:1}
\begin{proof}[Proof of Proposition \ref{prop:1}]
Put $g(u):=\int_{\log u}^{\infty}f(z)\,dz$.
It is true that $u^{\lambda}g(u)$ is slowly varying under the condition \eqref{neweqn:1}.
%because the tail integral of $f$ satisfies the condition \asptwo. 
In the following discussion, we take $u=e^{x}$.

For $\alpha\in\mathbb{R}$ and $x>1$, by change of variable and Fubini's theorem, we have
\begin{align*}
\int_{x}^{\infty}(e^{\alpha y}-&\ e^{\alpha x})f(y)\,dy
= \alpha \int_{x}^{\infty} \int_{x}^{y} e^{\alpha z} f(y)\,dzdy\\
=&\ \alpha \int_{x}^{\infty} e^{\alpha z} \int_{z}^{\infty}f(y)\,dydz
%=\alpha \int_{e^{x}}^{\infty} y^{\alpha-1}g(y)\,dy
=\alpha \int_{u}^{\infty}y^{\alpha-1}g(y)dy.
\end{align*}
Applying Proposition 1.5.1 of \cite{Bingham1987:book},
the last integral converges if $\alpha-\lambda<0$ and diverges if $\alpha-\lambda>0$, which proves those results on the finiteness of integrals $\int_{1}^{\infty}e^{\alpha y}\int_{y}^{\infty}f(z)\,dzdy$ and $\int_{1}^{\infty}e^{\alpha y}f(y)\,dy$ in \eqref{a} and \eqref{b}, respectively.

If $\lambda>0$ and $\alpha-\lambda<0$, applying Karamata's theorem, 
%{\color{red} c.f. Theorem 1.5.11 of \cite{Bingham1987:book}},
we further have
\[
\alpha \int_{u}^{\infty} y^{\alpha-1}g(y)\,dy\sim \frac{\alpha}{\lambda-\alpha} \big(u^{\alpha}g(u)\big)=\frac{\alpha}{\lambda-\alpha} e^{\alpha x}\int_{x}^{\infty}f(y)\,dy
\]
which proves the last result of \eqref{b}.

If $\lambda=0$ and $\alpha>0$, then $u^{\alpha}g(u)\to\infty$ as $u\to\infty$. By integration by parts, we have
\begin{align*}
\int_{1}^{x}e^{\alpha y}f(y)\,dy
=&\ e^{\alpha}\int_{1}^{\infty}f(y)\,dy-e^{\alpha x}\int_{x}^{\infty}f(y)\,dy
+ \alpha\int_{1}^{x}e^{\alpha y}dy \int_{y}^{\infty}f(z)\,dz\\
=&\ e^{\alpha}g(e)- \Big(u^{\alpha }g(u)- \alpha \int_{e}^{u}z^{\alpha-1}g(z)\,dz\Big).
\end{align*}
Since $u^{\alpha }g(u)\sim \alpha \int_{e}^{u}z^{\alpha-1}g(z)\,dz$ by Karamata's theorem, the last result of \eqref{a} holds.

If $\lambda>0$, for any $\varepsilon>0$, by Proposition 1.5.1 of \cite{Bingham1987:book},
we have $\D u^{\lambda+\varepsilon}g(u)\to\infty$ and $\D u^{\lambda-\varepsilon}g(u)\to0+$ as $u\to\infty$. Since $g(u)$ is continuous and decreasing, for any fixed $M>0$ we have $g(e^{k(x)})=x$ and
\[e^{(\lambda-\varepsilon)k(x)}x\leq M^{-1}
\quad\text{and}\quad
e^{(\lambda+\varepsilon)k(x)} x\geq M\]
for all small enough $x>0$.
Therefore,
\[\frac{-1}{\lambda+\varepsilon}\big(\log x-\log M\big)\leq k(x)\leq \frac{-1}{\lambda-\varepsilon}\big(\log x+\log M\big),\]
 which leads to the result of \eqref{c}.
\end{proof}

Applying Lemma \ref{lem:1} and Proposition \ref{prop:1} above, we first obtain the following asymptotic result on the integral with respect to potential measure $U$ of function $f$ which satisfies \eqref{neweqn:1} with $\lambda=0$.
%%%% Lemma 1 and 2
%%%% Lemma 1 and 2
%%%% Lemma 1 and 2
\begin{lem}\label{lem:2}
Suppose that $\gamma\in(0,\infty)$
and $f\geq0$ is an integrable function on $(0,\infty)$ with
\[\D \int_{x+a}^{\infty}f(y)\,dy\sim \int_{x}^{\infty}f(y)\,dy\quad\text{ for every}\quad a>0.\] Then
\[ \int_{0}^{\infty}f(y)U(x,dy)\sim \gamma^{-1}\int_{x}^{\infty}f(y)\,dy
\quad\text{as $x\to\infty$}.\]
\end{lem}
%%%%%[Proof of Lemma \ref{lem:2}]
%%%%%[Proof of Lemma \ref{lem:2}]
%%%%%[Proof of Lemma \ref{lem:2}]
\begin{proof}[Proof of Lemma \ref{lem:2}]
Notice that
$W(y-x)=0$ for $y<x$, we have from \eqref{eqn:resolvent} that for $x>0$,
\begin{align*}
&\ \int_{0}^{\infty}f(y)u(x,y)\,dy\\
=&\int_{0}^{\infty}\big(e^{-px}W(y)-W(y-x)\big)f(y) dy\\
=&\ e^{-px}\int_{0}^{x}W(y)f(y) dy+
\int_{0}^{\infty}\big(e^{-px}W(x+y)-W(y)\big)f(x+y) dy\\
=&:I_{1}+I_{2}.
\end{align*}

Since $W(y)e^{-py}=W_{p}(y)\uparrow \Phi'(0)<\infty$ as $y\rightarrow\infty$, we have from Proposition \ref{prop:1}\eqref{a} that
\[
I_{1}\leq \Phi'(0) e^{-px}\int_{0}^{x}e^{py}f(y)\,dy = o\big(\int_{x}^{\infty}f(y)\,dy\big).
\]
On the other hand, for every $\varepsilon>0$, applying Lemma \ref{lem:1},
for some $k>0$,
\[
\Big|\big(e^{-px}W(x+y)-W(y)\big)-\gamma^{-1}\Big|\leq \varepsilon \gamma^{-1},
\quad\text{for $x,y>k$}.
\]
Since $e^{-px}W(x+y)-W(y)\leq \Phi'(0)$ by \eqref{eqn:lem1:1}, then for $x>k$
\begin{align*}
 \Big|I_{2}- \gamma^{-1}& \int_{x}^{\infty}f(y)\,dy \Big|\\
\leq& (\Phi'(0)+\gamma^{-1}) \int_{0}^{k}f(x+y)\,dy+
\varepsilon \gamma^{-1} \int_{k}^{\infty}f(x+y)\,dy\\
=&\ (\Phi'(0)+\gamma^{-1}-\varepsilon \gamma^{-1}) \big(\int_{x}^{\infty}f(y)\,dy- \int_{x+k}^{\infty}f(y)\,dy\big)+ \varepsilon \gamma^{-1} \int_{x}^{\infty}f(y)\,dy\\
\sim&\ \varepsilon \gamma^{-1}\cdot \int_{x}^{\infty}f(y)\,dy
\quad\text{as $x\to\infty$,}
\end{align*}
where we used the assumption that the function of tail integral $x\to \int_{\log x}^{\infty}f(y)\,dy$ is slowly varying.
This finishes the proof.
\end{proof}

We are now ready to prove part \eqref{case:thm1:1} of Theorem \ref{thm:1}. 
Denote
\[J(\tau_{x}^{+}):=\int_{\tau_{x}^{+}}^{\infty}\omega(\xi_{t})\,dt=T_{\infty}-T^{+}_{x}
\quad\text{on the set $\{\tau_{0}^{-}=\infty\}$}.\]
We start with investigating the asymptotic behaviors of the first two moments of $\eta(\infty)$ under $\bP_{x}$, and then estimate 
the fist two moments of $J(\tau_{x}^{+})$ under $\bP_{1}$  using the Markov property of $\xi$.
Recall that $\varphi$ is the tail integral defined in \eqref{def_phi} and $p,\gamma\in(0,\infty)$ in \eqref{eqn:defn:c}.

%[Proof of Theorem \ref{thm:1}]
%[Proof of Theorem \ref{thm:1}]
%[Proof of Theorem \ref{thm:1}]
\begin{proof}[Proof of Theorem \ref{thm:1}\,\eqref{case:thm1:1}]
In this case, $\omega$ has a finite tail integral and satisfies condition \eqref{neweqn:1} with $\lambda=0$. 
In the following moment argument we further assume that
\[\D \varphi(0)=\frac{1}{\gamma}\int_{0}^{\infty}\omega(y)\,dy<\infty,\]
 under which we have $m_{2}(x)<\infty$ by Theorem \ref{thm:4}.
In case the above assumption does not hold, we can first prove the convergence result under
$\bq_{x}\big( \cdot\big|T_{\infty}^{+}<T_{c}^{-}=\infty\big)=\bP_{x}\big(\cdot\big|\tau_{c}^{-}=\infty\big)$ for $c>0$, and then let $c\rightarrow 0+$ to obtain the desired result.

Recall the moments $m_{1}$ and $m_{2}$ in Proposition \ref{prop:2},
\begin{align}
m_{1}(x)=&\ \int_{0}^{\infty}
u(x,y)\omega(y)(1-e^{-py})\,dy,
\label{eqn:prop2:1}\\
m_{2}(x)=&\ 2\int_{0}^{\infty}
u(x,y)\omega(y)m_{1}(y)\,dy.
\label{eqn:prop2:2}
\end{align}
We first claim that, as $x\to\infty$,
\begin{equation}\label{eqn:casea:1}
m_{1}(x)\sim \varphi(x),\quad
m_{2}(x)\sim \varphi^{2}(x),
\end{equation}
and in addition,
\begin{equation}\label{eqn:casea:2}
\bE_{1}\big(h(\xi(\tau_{x}^{+})); \tau_{x}^{+}<\tau_{0}^{-}\big)\sim
\bP_{1}(\tau_{0}^{-}=\infty)g(x)
\end{equation}
where $h\sim g$ and $g$ is a decreasing function $g$ such that $x\to g(\log x)$ varies slowly at $\infty$.
%if $h\sim g$ for some decreasing function $g$ such that $g(\log x)$ varies slowly at $\infty$.

Given \eqref{eqn:casea:1} and \eqref{eqn:casea:2},
since $J(\tau_{x}^{+})=\eta(\infty)\circ\theta_{\tau_{x}^{+}}$,
we have for $x>1$,
\begin{align*}
\bE_{1}^{\uparrow}\big(J(\tau_{x}^{+})\big)
=&\ \frac{1}{\bP_{1}(\tau_{0}^{-}=\infty)}
\bE_{1}\big(m_{1}(\xi(\tau_{x}^{+})); \tau_{x}^{+}<\tau_{0}^{-}\big)
\sim \varphi(x),\\
\bE_{1}^{\uparrow}\big(J^{2}(\tau_{x}^{+})\big)
=&\ \frac{1}{\bP_{1}(\tau_{0}^{-}=\infty)}
\bE_{1}\big(m_{2}(\xi(\tau_{x}^{+})); \tau_{x}^{+}<\tau_{0}^{-}\big)
\sim \varphi^{2}(x),
\end{align*}
which implies that
\[\D \bE_{1}^{\uparrow}\Big(\big(\frac{J(\tau_{x}^{+}\big)}{\varphi(x)}-1\big)^{2}\Big)
\underset{x\to\infty}{\longrightarrow} 0\]
and the desired weak convergence follows.

%1%1%1%1
To prove \eqref{eqn:casea:1},
we apply Lemma \ref{lem:2} to the function $f_{1}(y)=\omega(y)(1-e^{-py})$.
It is not hard to see that
\[\D\int_{x}^{\infty}f_{1}(y)\,dy\sim \int_{x}^{\infty}\omega(y)\,dy
\sim \int_{x+a}^{\infty}\omega(y)dy \sim \int_{x+a}^{\infty}f_{1}(y)dy.\]
Thus, 
%by the assumptions of Theorem \ref{thm:1}, 
$f_{1}$ fulfills the condition in Lemma \ref{lem:2}.
It follows from \eqref{eqn:prop2:1} that
\[m_{1}(x)=\int_{0}^{\infty}u(x,y) f_{1}(y)\,dy\sim \gamma^{-1}\int_{x}^{\infty}f_{1}(y)\,dy
\sim \varphi(x)\quad\text{as $x\to\infty$}.\]
%2%2%2%2
Then, we take $f_{2}(y)=\omega(y)m_{1}(y)$.
 From the result above, for any $\varepsilon\in(0,1)$ let $k_{1}>0$ satisfy
\[(1-\varepsilon)\varphi(x)\leq
m_{1}(x)\leq
(1+\varepsilon)\varphi(x)
\quad\text{for $x>k_{1}$}.\]
It follows that for $x>k_{1}$,
\[\begin{gathered}
\int_{x}^{\infty}f_{2}(y)\,dy
\leq (1+\varepsilon) \int_{y>x}\omega(y)\varphi(y)\,dy
=\frac{1+\varepsilon}{2} \gamma\varphi^{2}(x),\\
\int_{x}^{\infty}f_{2}(y)\,dy
\geq (1-\varepsilon) \int_{y>x}\omega(y)\varphi(y)\,dy
=\frac{1-\varepsilon}{2} \gamma\varphi^{2}(x),
\end{gathered}\]
which gives that
\[\D\int_{x}^{\infty}f_{2}(y)\,dy\sim \frac{1}{2}\gamma\varphi^{2}(x),\]
 and $f_2$ satisfies the condition of Lemma \ref{lem:2}.
Applying \eqref{eqn:prop2:2} and Lemma \ref{lem:2} we have
\[m_{2}(x)= 2\int_{0}^{\infty}f_{2}(y)u(x,y)\,dy\sim 2\gamma^{-1}\int_{x}^{\infty}f_{2}(y)\,dy
\sim \varphi^{2}(x).\]

%%To prove \eqref{eqn:casea:2}
%%To prove \eqref{eqn:casea:2}
%%To prove \eqref{eqn:casea:2}
To prove \eqref{eqn:casea:2}, let $k_{2}>0$ satisfy
\[
\rho([0,k_{2}])\geq 1-\varepsilon\quad\text{and}\quad
(1-\varepsilon)g(x)\leq h(x)\leq (1+\varepsilon)g(x)
\quad\text{for $x>k_{2}$},
\]
where $\rho$ is the stationary overshoot distribution in \eqref{eqn:defn:rho}.
Then for $x>k_{2}$,
\begin{align*}
(1+\varepsilon) \geq&\ \frac{\bE\big(h(\xi_{\tau_{x}^{+}})\big)}{g(x)}
\geq (1-\varepsilon)
\frac{\bE\big(g(\xi_{\tau_{x}^{+}})\big)}{g(x)}\\
\geq&\ (1-\varepsilon)
\frac{\bE\big(g(\xi_{\tau_{x}^{+}}); \xi_{\tau_{x}^{+}}<x+k_{2})\big)}{g(x)}
\geq (1-\varepsilon)^{2} \frac{g(x+k_{2})}{g(x)}
\end{align*}
where the monotonicity of $g$ is applied to the first and the last inequality. Thus,
\[
\bE\big(h(\xi_{\tau_{x}^{+}})\big)\sim g(x)\quad \text{as $x\to\infty$}.
\]
Lastly, applying the strong Markov property for $\xi$ we further have
\begin{align*}
&\ \bE_{1}\big(h(\xi_{\tau_{x}^{+}}); \tau_{x}^{+}<\tau_{0}^{-}\big)\\
=&\ \bE_{1}\big(h(\xi_{\tau_{x}^{+}})\big)- \bP_{1}\big(\tau_{0}^{-}<\tau_{x}^{+}\big)\cdot \bE\big(h(\xi_{\tau_{x}^{+}})\big)\\
=&\ \bE\big(h(\xi_{\tau_{x-1}^{+}}+1)\big)- \bE\big(h(\xi_{\tau_{x}^{+}})\big)+
\bP_{1}(\tau_{x}^{+}<\tau_{0}^{-})\cdot \bE\big(h(\xi_{\tau_{x}^{+}})\big)\\
\sim&\ \bP_{1}(\tau_{x}^{+}<\tau_{0}^{-})\cdot g(x) \sim \bP_{1}(\tau_{0}^{-}=\infty) g(x) \quad\text{as}\quad x\rightarrow\infty.
\end{align*}
This finishes the proof.
\end{proof}
%%% remark
%%% remark
%%% remark
\begin{rmk}
In the proof of statement \eqref{case:thm1:1}, we have
\[
\bE^{\uparrow}_{1}\big(J(\tau_{x}^{+})\big)
\sim \bE^{\uparrow}_{x}\big(\eta(\infty)\big)
\sim \varphi(x)
\quad\text{as $x\to\infty$}.
\]
However, with the presence of positive jumps, $m_{1}(x)$ and $\bE^{\uparrow}_{x}\big(\eta(\infty)\big)$ may fail to be monotone in $x$ in general.
\end{rmk}

%% For the proof of case \ref{case:thm1:2}
%% For the proof of case \ref{case:thm1:2}
%% For the proof of case \ref{case:thm1:2}
For the proof of statement \eqref{case:thm1:2} of Theorem \ref{thm:1}, we make use of the local time for the process $\xi$, see
c.f. Chapter V of \cite{Bertoin96:book} for more detailed discussion.
Given a SPLP $\xi$, its local time is well-defined
and defined as the density of occupation measure by, $\bP$-a.s.,
\[L(y,t):=\lim_{\varepsilon\to0+}\frac{1}{2\varepsilon}\int_{0}^{t}\mathbf{1}(|\xi_{s}-y|<\varepsilon)\,ds,\quad\text{for $y\in\mathbb{R}, t>0$}.\]
and the following \textit{occupation density formula} holds for all measurable bounded function $f\geq0$,
\[
\int_{0}^{t}f(\xi_{s})\,ds=\int_{\mathbb{R}}f(y) L(y,t)\,dy
\quad\text{$\bP$-a.s}.
\]
We also need the following lemmas on the regularly varying functions and the local time.

%\label{lem:3}\label{lem:4}
%\label{lem:3}\label{lem:4}
%\label{lem:3}\label{lem:4}
\begin{lem}\label{lem:3}
Let $\omega$ be the function in statement \eqref{case:thm1:2} of Theorem \ref{thm:1}.
Let $f\geq0$ be a measurable function locally integrable such that the set $\{x\in\mathbb{R}, f(x)>0\}$ is bounded from below,
 \[\D \int_{\mathbb{R}} e^{-\lambda y}f(y)\,dy<\infty\quad
\text{and}\quad \D \int_{\mathbb{R}} e^{-2\alpha y}f^{2}(y)\,dy<\infty\quad
\text{for some}\quad 2\alpha\in(0,\lambda).\] Then
\begin{equation}\label{eqn:lem3:1}
\frac{1}{\varphi(x)}\int_{\mathbb{R}}\omega(x+y)f(y)\,dy
\underset{x\to\infty}{\longrightarrow}
\lambda\gamma \int_{\mathbb{R}}e^{-\lambda y}f(y)\,dy.
\end{equation}
\end{lem}
\begin{lem}\label{lem:4}
Suppose that $p>0$. For any $2\alpha\in(0,p)$, we have
\[
\bE\big(\int_{\mathbb{R}} e^{-2\alpha y}L^{2}(y,\infty)\,dy\big)
<\infty.
\]
\end{lem}
% rmk
% rmk
% rmk
\begin{rmk}
 Lemma \ref{lem:3} appears similar to the
Abelian theorem, c.f. Theorem 4.1.3 of \cite{Bingham1987:book}
where $\omega(\log\cdot)$ is assumed to be regularly varying,
and also similar to Theorem 1.7.5 in \cite{Bingham1987:book}, where conditions related to \textit{slowly decreasing} is imposed. The condition here can be replaced by other, possibly weaker, conditions. For example, if $f$ has bounded variation and is bounded, right-continuous, and $\{x\in\mathbb{R}, f(x)>0\}$ is bounded from below, an application of the uniform converge theorem could give the same result.
\end{rmk}
%% [Proof of Lemma \ref{lem:3}]
%% [Proof of Lemma \ref{lem:3}]
%% [Proof of Lemma \ref{lem:3}]
\begin{proof}[Proof of Lemma \ref{lem:3}]
Since the set $\{x\in\mathbb{R}, f(x)>0\}$ is bounded from below,
it is sufficient to prove \eqref{eqn:lem3:1} for $f$ vanishing on $(-\infty,0)$,
and we only focus on integrals on $(0,\infty)$.

By the assumption in statement \eqref{case:thm1:2}, for some $K,M>0$, we have
\[
\int_{x}^{\infty}\omega^{2}(y)\,dy\leq M\cdot \varphi^{2}(x)
\quad \text{for all $x>K$}.
\]
Applying Fubini's theorem, for $x>K$ we have
\begin{align*}
&\ \int_{0}^{\infty} \big(e^{2\alpha y}- 1\big) \omega^{2}(x+y)\,dy
=\int_{0}^{\infty}\omega^{2}(x+y)\,dy \int_{0}^{y} 2\alpha e^{2\alpha t}\,dt\\
=&\ 2\alpha \int_{0}^{\infty} e^{2\alpha t}dt \int_{x+t}^{\infty} \omega^{2}(y)\,dy
\leq 2 \alpha M \int_{0}^{\infty} e^{2\alpha t} \varphi^{2}(x+t)\,dt\\
=&\ 2\alpha M \int_{x}^{\infty} e^{2\alpha(t-x)}\varphi^{2}(t)\,dt
=2\alpha M \Big(u^{-2\alpha }\int_{u}^{\infty} s^{2\alpha-1}\varphi^{2}(\log s)ds\Big)\Big|_{u=e^{x}}\\
\sim&\ \frac{\alpha\cdot M}{\lambda-\alpha} \varphi^{2}(\log u)\Big|_{u=e^{x}}=
\frac{\alpha\cdot M}{\lambda-\alpha}\cdot \varphi^{2}(x) \quad\text{as $x\to\infty$},
\end{align*}
by applying Karamata's theorem to the last line since 
$s\to \varphi^{2}(\log s)$ is regularly varying with index $-2\lambda$.
Thus, under the assumption for Theorem \ref{thm:1} \eqref{case:thm1:2},
\[\limsup_{x\rightarrow \infty}\int_{0}^{\infty} e^{2\alpha y}\frac{\omega^{2}(x+y)}{\varphi^{2}(x)}\,dy<\infty. \]
The Cauchy-Schwarz inequality then yields
\begin{equation}\label{holder}
\begin{split}
\big(\int_{0}^{\infty}\frac{\omega(x+y)}{\varphi(x)}f(y)\,dy\big)^{2}
&\leq
\Big(\int_{0}^{\infty} e^{-2\alpha y}f^{2}(y)\,dy\Big)
\cdot
\Big(\int_{0}^{\infty} e^{2\alpha y}\frac{\omega^{2}(x+y)}{\varphi^{2}(x)}\,dy\Big),
\end{split}
\end{equation}
where the second term on the right hand side is dominated by some constant.

Firstly, the limit \eqref{eqn:lem3:1} holds for simple functions $f(x)=\mathbf{1}(x>c), \forall c\geq0$ as well as their linear combinations.
Moreover, it holds for any bounded measurable function $f$ on $(0,\infty)$ which can be uniformly and non-decreasingly approximated by simple functions $f_{n}$ satisfying \eqref{eqn:lem3:1}.
Therefore, it holds for all nonnegative bounded Borel functions on $(0,\infty)$,
by applying the functional monotone converge theorem, c.f. Theorem 2.12.9. \cite{Bogachev:book}.
Finally, for any function $f$ satisfying the assumption of Lemma \ref{lem:3}, taking $f_{n}=f\wedge n$ 
and applying \eqref{holder} gives 
\[\lim_{n\to\infty}\limsup_{x\to\infty}\Big(\int_{0}^{\infty}\frac{\omega(x+y)}{\varphi(x)}\big(f(y)-f_{n}(y)\big)dy\Big)^{2}
=0.\]
It follows that 
\begin{align*}
 \lim_{x\to\infty}\int_{0}^{\infty}&\frac{\omega(x+y)}{\varphi(x)}f(y)dy
=\lim_{n\to\infty}\lim_{x\to\infty}\int_{0}^{\infty}\frac{\omega(x+y)}{\varphi(x)}f_{n}(y)dy\\
=&\lim_{n\to\infty}\lambda\gamma\int_{0}^{\infty}e^{-\lambda y}f_{n}(y)dy
=\lambda\gamma \int_{0}^{\infty} e^{-\lambda y}f(y)dy.
\end{align*}
This finishes the proof.
\end{proof}

Lemma \ref{lem:4} is proved following the argument used in Theorem V.1 of \cite{Bertoin96:book}, where Plancherel's theorem is applied.
%[Proof of Lemma \ref{lem:4}]
%[Proof of Lemma \ref{lem:4}]
%[Proof of Lemma \ref{lem:4}]
\begin{proof}[Proof of Lemma \ref{lem:4}]
Let $2\alpha\in(0,p)$ and $g(y)=e^{-\alpha y}L(y,\infty)$ for $y\in\mathbb{R}$.
Then $\psi(\alpha),\psi(2\alpha)<0$ by definition.
Applying Fubini's theorem and the occupation density formula, we have
\[\bE\big(\int_{0}^{\infty} e^{-2\alpha \xi_{t}}\,dt\big)
= \int_{0}^{\infty}\bE\big(e^{-2\alpha \xi_{t}}\big)\,dt=\frac{-1}{\psi(2\alpha)}\]
and
\[\bE\big(\int_{\mathbb{R}} e^{-\alpha y} L(y,\infty)\,dy\big)
=\bE\Big(\int_{0}^{\infty}e^{-\alpha \xi_{t}}\,dt\Big)=\frac{-1}{\psi(\alpha)}.\]
Thus, $e^{-\alpha\xi_{t}}$, $e^{-2\alpha\xi_{t}}$ and $g(y)$ are all integrable.
The Fourier transform of $g$ gives for every $u\in\mathbb{R}$,
\[\mathcal{F}g(u)=\int_{\mathbb{R}} e^{iuy}e^{-\alpha y} L(y,\infty)\,dy=\int_{0}^{\infty} e^{(iu-\alpha)\xi_{t}}\,dt,\quad \bP\text{-a.s.}\]
 In addition, we have
\begin{align*}
 \bE\big(|\mathcal{F}g(u)|^{2}\big)=&\ \bE\big(\mathcal{F}g(u)\mathcal{F}g(-u)\big)
= \bE\Big(\int_{0}^{\infty}\int_{0}^{\infty} e^{(iu-\alpha)\xi_{t}+(-iu-\alpha)\xi_{s}}\,dtds\Big)\\
=&\ \bE\Big(\int_{0}^{\infty}\int_{0}^{\infty} e^{(iu-\alpha)(\xi_{t}-\xi_{s})-2\alpha\xi_{s}}\,dtds\Big).
\end{align*}
Given the integrability of $e^{-\alpha\xi_{t}}$ and $e^{-2\alpha \xi_{t}}$, we can re-express the last term as
\[2\Re\Big(\int_{0}^{\infty}\,ds\int_{s}^{\infty} \,dt\bE\big(e^{(iu-\alpha)(\xi_{t}-\xi_{s})}\big) \bE\big(e^{-2\alpha\xi_{s}}\big)\Big).\]
Under the new measure $\bP^{(\alpha)}$ the above quantity equals to
\[2\Re\Big(\int_{0}^{\infty}e^{\psi(\alpha)t} \bE\big(e^{iu\xi_{t}} e^{-\alpha\xi_{t}-\psi(\alpha)t} \big) \,dt
\int_{0}^{\infty} \bE\big(e^{-2\alpha\xi_{s}}\big)\,ds\Big)
= \frac{-2}{\psi(2\alpha)}\Re\Big(\frac{1}{\Psi_{\alpha}(u)-\psi(\alpha)}\Big),\]
where
$$\D\Psi_{\alpha}(s)=t^{-1}\log\bE^{(\alpha)}\big(e^{is\xi_{t}}\big)=-\psi_{\alpha}(-i s)$$
 is the characteristic exponent of $\xi$ under $\bP^{(\alpha)}$.
Noticing that $-\psi(\alpha)>0$, we have
\[\int_{\mathbb{R}} \bE\big(|\mathcal{F}g(u)|^{2}\big)\,du=\frac{-2}{\psi(2\alpha)}
\int_{\mathbb{R}} \Re\Big(\frac{1}{\Psi_{\alpha}(u)-\psi(\alpha)}\Big)\,du<\infty,\]
where Theorem II.16 in \cite{Bertoin96:book} is applied. The proof is finished by applying Plancherel's theorem.
\end{proof}

Now, we are ready to prove the result of part \eqref{case:thm1:2}.
%Proof of case \ref{case:thm1:2}
%Proof of case \ref{case:thm1:2}
%Proof of case \ref{case:thm1:2}
\begin{proof}[Proof of Theorem \ref{thm:1}\,\eqref{case:thm1:2}]
Recall that $p,\gamma\in(0,\infty)$.
Let $\lambda>0$ be the constant  in  condition \asptwo, 
and $f,g$ be bounded continuous and nonnegative functions.

% and the condition \asptwo\ holds for $\omega$ for some $\lambda>0$. 

Applying the strong Markov property of $\xi$ at $\tau_{x}^{+}$, we have
\begin{equation}\label{eqn:caseb:1}
\bE_{1}^{\uparrow}\Big(
f\big(\frac{\varphi(\xi(\tau_{x}^{+}))}{\varphi(x)}\big)
\cdot
g\big(\frac{J(\tau_{x}^{+})}{\varphi(\xi(\tau_{x}^{+}))}\big)\Big)
= \bE^{\uparrow}_{1}\Big(
f\big(\frac{\varphi(\xi(\tau_{x}^{+}))}{\varphi(x)}\big)\cdot G(\xi(\tau_{x}^{+}))\Big).
\end{equation}
Denote by $\mathring{\xi}$ an independent copy of $\xi$ with probability law of $\mathring{\bP}$ and define for $z>0$
\[G(z):=\mathring{\bE}\Big(
g\Big(\int_{0}^{\infty}
\frac{\omega(\mathring{\xi}_{t}+z)}{\varphi(z)}\,dt\Big)\Big|
\mathring{\tau}_{-z}^{-}=\infty\Big).\]

Let $\mathring{L}(y,t)$ be the local time of $\mathring{\xi}$ at level $y$ and time $t$.
Since $\mathring{\xi}(t)\to\infty$,
we have from \eqref{eqn:ky} that $\D \int_{0}^{\infty}e^{-\lambda \mathring{\xi}_{t}}dt<\infty$ 
$\mathring{\bP}$-a.s.
Applying Theorem I.20 of \cite{Bertoin96:book}, Lemma \ref{lem:4}
and the fact that $\D \big|\inf_{t>0}\mathring{\xi}_{t}\big|<\infty$,
one can check that, for $2\alpha<\lambda \wedge p$,
$\mathring{L}(y,\infty)$ fulfills the conditions of Lemma \ref{lem:3} $\mathring{\bP}$-a.s..
 Therefore,
\begin{align*}
&\ \int_{0}^{\infty}
\frac{\omega(\mathring{\xi}_{t}+z)}{\varphi(z)}\,dt
=\int_{\mathbb{R}}\frac{\omega(z+y)\mathring{L}(y,\infty)}{\varphi(z)}\,dy\\
\underset{z\to\infty}{\longrightarrow}&\ \lambda\gamma \int_{\mathbb{R}}e^{-\lambda y}\mathring{L}(y,\infty)\,dy
=\lambda\gamma \int_{0}^{\infty} e^{-\lambda \mathring{\xi}_{t}}\,dt,
\quad \text{$\mathring{\bP}$-a.s}.
\end{align*}
Moreover, since $\mathbf{1}(\mathring{\tau}_{-z}^{-}=\infty)\to 1$ $\mathring{\bP}$-a.s. as $z\to\infty$,
by the dominated convergence theorem,
\[
G(z)=\frac{1}{1-e^{-pz}}
\mathring{\bE}\Big(g\big(\int_{0}^{\infty}
\frac{\omega(\mathring{\xi}_{t}+z)}{\varphi(z)}\,dt\big)\cdot \mathbf{1}(\mathring{\tau}_{-z}^{-}=\infty)\Big)\\
\underset{z\to\infty}{\longrightarrow}
\mathring{\bE}\Big( g\big(\lambda\gamma \int_{0}^{\infty} e^{-\lambda \mathring{\xi}_{t}}\,dt\big)\Big).
\]

On the other hand, by the uniform converge theorem for $\varphi$, see
Theorem 1.5.2 of \cite{Bingham1987:book}, we have
$\D {\varphi(x+y)}/{\varphi(x)}\to e^{-\lambda y}$ as $x\to\infty$,
uniformly for $y\in[0,\infty)$.

Applying \eqref{eqn:defn:rho} and the facts above to \eqref{eqn:caseb:1}, we complete the proof.
\end{proof}

%\tb{To ....., we first investigate the sample path of $\xi$ at infinity in Lemma \ref{lem:5},
%and Theorem \ref{thm:2} is then proved following the same idea of \cite{paper}}

%\begin{proof}[Proof of Theorem \ref{thm:2}]
%\begin{proof}[Proof of Theorem \ref{thm:2}]
%\begin{proof}[Proof of Theorem \ref{thm:2}]
For the proof of Theorem \ref{thm:2}, we follow the same idea from \cite{Foucart2019} and \cite{Bansaye16}.
\begin{proof}[Proof of Theorem \ref{thm:2}]
The theorem is proved by first claiming that under $\bP_{1}^{\uparrow}$,
\begin{equation}\label{eqn:thm2:2}
\frac{\bar{X}(T^{+}_{\infty}-t)}{\varphi^{-1}(t)}
\stackrel{D}{\Longrightarrow} 1
\quad\text{as $t\to0+$},
\end{equation}
recalling that $\bar{X}(t):=\sup_{s\in[0,t]}$ represents the running maximum of $X$.
The desired conclusion then follows from Lemma \ref{lem:5}.

We first prove statement \eqref{case:thm2:1}. 
For any $h>1$, we take a constant $\D c_{h}\in \Big(1,\liminf_{y\to\infty}\frac{\varphi(y)}{\varphi(hy)}\Big)$.
By the result of statement \ref{case:thm1:1} of Theorem \ref{thm:1}, for any $\varepsilon>0$
there is $k_{3}>0$ such that for $x>k_{3}$,
\begin{equation}\label{eqn:thm2:3}
\bP_{1}^{\uparrow}
\Big(\frac{J(\tau_{x}^{+})}{\varphi(x)}
\notin \big(c_{h}^{-1}, c_{h}\big)\Big)
\leq \varepsilon
\quad\text{and}\quad
1<c_{h}\leq \frac{\varphi(x)}{\varphi(hx)}.
\end{equation}
Set $t_{0}:=\varphi(h k_{3})$.
For $t<t_{0}$, define $\alpha(t):=h\cdot \varphi^{-1}(t)$ and $\beta(t):=\varphi^{-1}(t)/h$, then $\alpha(t)>\varphi^{-1}(t)>\beta(t)>k_{3}$. By the second inequality in \eqref{eqn:thm2:3},
\[
\frac{\varphi(\varphi^{-1}(t))}{\varphi(\beta(t))}\leq c_{h}^{-1}<
1<c_{h}\leq \frac{\varphi(\varphi^{-1}(t))}{\varphi(\alpha(t))}.
\]
Then we further have from $\beta(t)>k_{3}$ that the following inequalities hold,
\begin{align*}
\bP_{1}^{\uparrow}\big(J(\tau_{\alpha(t)}^{+})\geq t\big)
=&\ \bP_{1}^{\uparrow}\Big(\frac{J(\tau_{\alpha(t)}^{+})}{\varphi(\alpha(t))}
\geq \frac{\varphi(\varphi^{-1}(t))}{\varphi(\alpha(t))}\Big)
\leq \bP_{1}^{\uparrow}\Big(\frac{J(\tau_{\alpha(t)}^{+})}{\varphi(\alpha(t))}\geq c_{h}\big)
\leq \varepsilon,\\
 \bP_{1}^{\uparrow}\big(J(\tau_{\beta(t)}^{+})\leq t\big)
=&\ \bP_{1}^{\uparrow}\Big(\frac{J(\tau_{\beta(t)}^{+})}{\varphi(\beta(t))}
\leq \frac{\varphi(\varphi^{-1}(t))}{\varphi(\beta(t))}\Big)
\leq \bP_{1}^{\uparrow}\Big(\frac{J(\tau_{\beta(t)}^{+})}{\varphi(\beta(t))}\leq c_{h}^{-1}\big)
\leq \varepsilon,
\end{align*}
 which, since $\varphi(\varphi^{-1}(t))=t$, gives for $t<t_{0}$
\[
\bP_{1}^{\uparrow}\big(h^{-1}\leq \frac{\bar{X}(T^{+}_{\infty}-t)}{\varphi^{-1}(t)}\leq h\big) \\
= \bP_{1}^{\uparrow}\big(J(\tau_{\alpha(t)}^{+})\leq t\leq J(\tau_{\beta(t)}^{+})\big)\geq 1-2\varepsilon.
\]
We can prove the weak limit \eqref{eqn:thm2:2} by first letting $\varepsilon\rightarrow 0+$ and then letting $h\rightarrow 1+$.

For statement \eqref{case:thm2:2},
notice that for every $k_{4}>0$ and large $x$
\[
\frac{\varphi(x)}{\varphi(hx)}\geq \frac{\varphi(x)}{\varphi(x+k_{4})}\to e^{\lambda k_{4}}
\quad\text{as $x\to\infty$}.
\]
Thus, we always have $\D \liminf_{y\to\infty} \frac{\varphi(y)}{\varphi(hy)}=\infty$ for all $h>1$ in this case. Since $\D {J(\tau_{x}^{+})}/{\varphi(x)}$ converges in law to a random variable on $(0,\infty)$,
there exist $M>1$ and $k_{5}>0$ such that
\[
\bP^{\uparrow}_{1}\Big(\frac{J(\tau_{x}^{+})}{\varphi(x)}
\notin (\frac{1}{M}, M)\Big)\leq \varepsilon
\quad\text{and}\quad
\quad 1< M \leq \frac{\varphi(x)}{\varphi(hx)}
\quad\text{for all $x>k_{5}$},
\]
which can be compared with \eqref{eqn:thm2:3}.
The same argument as in the previous case can be applied to prove \eqref{eqn:thm2:2}.
Applying the result of \eqref{c} in Proposition \ref{prop:1}, we finish the proof.
\end{proof}

%\end{document}

%\bibliographystyle{apalike}
%\bibliography{/Users/sday_oce/Desktop/working/library}
%

\end{document}